\documentclass{amsart}
\usepackage{setspace}
\numberwithin{equation}{section}
\usepackage{graphicx}
\usepackage{caption}
\usepackage{subcaption}
\usepackage{amssymb}
\usepackage{enumerate, xspace}
\usepackage{dsfont}
\usepackage{tikz}
\usepackage{todonotes}
\usepackage{lscape}
\usetikzlibrary{decorations.pathreplacing,calligraphy}
\usepackage[bottom,first]{draftcopy}
\usepackage{lipsum}
\usepackage[normalem]{ulem}
\usepackage[colorlinks]{hyperref}
\usepackage[graphicx]{realboxes}

\usepackage{siunitx}
\usepackage{datax}
\newtheorem{theorem}{Theorem}[section]

\newtheorem{definition}[theorem]{Definition}

\newtheorem{lemma}[theorem]{Lemma}
\newtheorem{sub-lemma}[theorem]{Sub-Lemma}

\newtheorem{remark}[theorem]{Remark}
\newtheorem{numassumption}[theorem]{Numerical assumption}

\newcommand\eps{\epsilon}

\newcommand{\norm}[1]{\left\lVert#1\right\rVert}

\newcommand{\marginnotet}[1]
           {\mbox{}\marginpar{\tiny\raggedright\hspace{0pt}{{\bf Toby}$\blacktriangleright$  {\color{blue} #1}}}}

\newcommand{\logg}{\text{logg}}

\begin{document}
\title{Rigorous Computation of Linear Response for intermittent maps}
\author{Isaia Nisoli}
\address{Instituto de Matematica - UFRJ Av. Athos da Silveira Ramos 149, Centro de Tecnologia - Bloco C Cidade Universitaria - Ilha do Fund\~ao. Caixa Postal 68530 21941-909 Rio de Janeiro - RJ - Brasil}
\email{nisoli@im.ufrj.br}
\author{Toby Taylor-Crush}
\address{Department of Mathematical Sciences, Loughborough University,
Loughborough, Leicestershire, LE11 3TU, UK}
\email{T.Taylor-Crush@lboro.ac.uk}
\date{\today }

\pagestyle{myheadings} 

\begin{abstract}
We present a rigorous numerical scheme for the approximation of the linear 
response of the invariant density of a map with an indifferent fixed point, 
with explicit and computed estimates for the error and all the involved constants.
\end{abstract}

\maketitle
\section{Introduction}
 
In \cite{R} Ruelle proved that for certain perturbations of uniformly hyperbolic 
deterministic dynamical systems the underlying SRB measure changes smoothly. 
He also obtained a formula for the derivative of the SRB measure, called the 
\emph{linear response formula} 
\cite{R}\footnote{See also earlier related work \cite{KP}. 
See also \cite{Got} for a comprehensive historical account including literature from physics.}.
Since then, the topic of linear response has been a very active direction of research in 
smooth ergodic theory. Indeed, the work of Ruelle was refined in the uniformly hyperbolic 
setting \cite{BL,GL}, extended to the partially hyperbolic setting \cite{D}, and has been 
a topic of deep investigation for unimodal maps, see \cite{Ba1}, the survey article \cite{Ba2}, 
the recent works \cite{ABLP, BaSm, LS, S} and references therein. 
More recently, the topic of linear response was also studied in the context of random or 
extended systems \cite{BSR, DS, GG, GS, KLP, ST, WG1}. 
Optimisation of statistichal properties through linear respone was develope in \cite{ADF, AFJ, GaPo, Klo}. 

Numerical algorithms for the approximation of linear response for uniformly expanding maps, 
via finite rank transfer operators was obtain in \cite{BGNN} and via dynamical determinants 
and periodic orbits in \cite{PV}\footnote{See also \cite{Ja} for related work on 
dynamical determinants.}, and for uniformly hyperbolic systems \cite{GLuc, Luc, Ni}. 

Our work extends the methods in \cite{BGNN} to intermittent maps far from the boundary, allowing
us to compute the linear response for LSV maps, a version
of the Manneville-Pomeau family, \cite{LSV} as the exponent at the indifferent fixed point changes.

Linear response for indifferent fixed point maps has been investigated in \cite{BS, BT, K}, 
but three important questions have to be addressed to obtain a rigorous numerical approximation scheme:
\begin{enumerate}
\item how to approximate efficiently the involved discretized operators;
\item how to bound the approximation errors involved in the discretization;
\item how to bound explicitly and efficiently the constants used in the proofs of \cite{BS, BT, K}.
\end{enumerate}

In our paper we provide answers to the three questions above for general intermittent maps 
and present an explicit computation for LSV type maps.  
Our scheme and tecniques are very flexible and can be easily 
adapted to other one dimensional nonuniformly expanding maps 
whose associated transfer operators do not admit a spectral gap (or a uniform spectral gap) 
as long as the linear response formula can be obtained via inducing with the first return map.

In the text are presented some numerical remarks, that allow the reader to get an overview
of some of the delicate points of the implementation.

The paper is divided as follows: in Section \ref{sec:hypothesis} we state the hypothesis on 
the dynamical system and state our results, in Section \ref{sec:density} we present the 
theory behind the approximation of the density for the induced map, in Section 
\ref{sec:InducedResponse} we discuss the approximation of the linear response for the induced
map, in Section \ref{sec:normalizing} we discuss pulling back the measure to the original map
and normalizing the density, in Section \ref{sec:mainthm} we give a proof of the fact that the error may be made as 
small as wanted, in Section \ref{sec:example} we compute an approximation with an explicit
error of the linear response for an LSV map; section \ref{sec:appendix} is devoted to computing
effective bounds for the constants in \cite{BS, K} and section \ref{sec:A0B0}
explains the tecnique we use to compute some of the functions involved in our approximation.

\section*{Acknowledgements}
The authors would like to thank Prof. Bahsoun and Prof. Galatolo for their guidance and
assistance, their patience and attention.
Isaia Nisoli was partially supported by CNPq, UFRJ, CAPES 
(through the programs PROEX and the CAPES-STINT project 
``Contemporary topics in non uniformly hyperbolic
dynamics'') and Hokkaido University.

\section*{Data Availability}
The package used for the computations may be found at \url{https://github.com/orkolorko/InvariantMeasures.jl}. The Jupyter notebook with the experiment can be provided under inquiry and will published online as soon as possible.

\section{Hypothesis on the map and statement of the results}\label{sec:hypothesis}

We are interested in approximating the invariant density and linear response 
for one dimensional interval maps with an indeterminate fixed point by inducing. 
In particular we wish to gain explicitly calculable error bounds in the $L^1$ norm. We use an induced map on $[0.5,1]$, to gain a map with good statistical properties to approximate an invariant density and linear response, and then using a formula used in \cite{BS} to pull back our approximation to the invariant density and linear response of the full map. We apply this method to a family of Pomeau-Manneville maps to gain an approximation of the statistics with explicit error.


\subsection{Interval maps with an inducing scheme}\label{ss:T}
We introduce now a class (family) of interval maps which are non-uniformly expanding with two branches,
for which one can construct an inducing scheme which allow it to inherit the linear response formula from 
the one for the induced system.

\begin{itemize}
\item Let $V$ be a neighbourhood of $0$. For any $\eps\in V$, $T_\eps\colon [0,1]\to[0,1]$ is a non-singular map, with respect to Lebesgue measure, $m$, with two onto branches $T_{0,\eps}\colon[0,0.5]\to[0,1]$ and $T_{1,\eps}\colon[0.5,1]\to[0,1]$. The inverse branches of  $T_{0,\eps}$, $T_{1,\eps}$ are respectively denoted by $g_{0,\eps}$ and $g_{1,\eps}$. We call $T_0:=T$ the unperturbed map, and $T_\eps$, for $\eps\not= 0$, the perturbed map.
\item We assume that for each $i=0,1$ and $j=0,1,2$ the following partial derivatives exist and satisfy the commutation relation
\begin{equation}\label{comm}
\partial_\eps g_{i,\eps}^{(j)} = (\partial_\eps g_{i,\eps})^{(j)}.
\end{equation}
\item We assume that $T_\eps$ has a unique absolutely continuous invariant measure\footnote{The $T_\eps$ absolutely continuous invariant measure is not assumed to be probabilistic; we allow for $T_\eps$ to admit a $\sigma$-finite absolutely continuous invariant measure.} (up to multiplication) whose Radon-Nikodym derivative will be denoted by $h_\eps$, and we denote for simplicity $h=h_0$.
\item Let $\hat T_\eps$, be the first return map of $T_\eps$ to $\Delta$, where $\Delta:= [0.5,1]$; i.e., for $x\in\Delta$
$$\hat T_{\eps}(x)=T_\eps^{R_\eps(x)}(x),$$
where 
$$R_\eps(x)=\inf\{n\ge 1:\, T^{n}_\eps(x)\in\Delta\}.$$
We assume that $\hat T_\eps$ has a unique absolutely continuous invariant measure (up to multiplication) with a continuous density denoted $\hat h_\eps\in C^0$.
\item Let $\Omega$  be the set of finite sequences of the form $\omega = 10^n$, for $n \in\mathbb N \cup\{0\}$. We set $g_{\omega,\eps}=g_{1,\eps}\circ g_{0,\eps}^{n}$. Then for $x\in [0,1]$ we have $T_{\eps}^{n+1}\circ g_{\omega,\eps}(x)=x$. The cylinder sets $[\omega]_{\eps} = g_{\omega, \eps}(\Delta)$, form a partition of $\Delta$ (mod $0$).
For $x\in [0,1]$, we assume
\begin{equation}\label{a4}
\sup_{\eps\in V}\sup_{x\in[0,1]}|g'_{\omega,\eps}(x)| <\infty ;
\end{equation}
\begin{equation}\label{a4.0}
\sup_{\eps\in V}\sup_{x\in[0,1]}|\partial_\eps g_{\omega,\eps}(x)| < \infty;
\end{equation}
\begin{equation}\label{a4'}
\sum_{\omega}\sup_{\eps\in V}||g'_{\omega,\eps}||_{\mathcal B} <\infty;
\end{equation}
and
\begin{equation}\label{a5}
\sum_\omega\sup_{\eps\in V}||\partial_\eps g_{\omega,\eps}'||_{\mathcal B}<\infty,
\end{equation}
\end{itemize}
where $\mathcal{B}$ denotes the set of continuous functions on $(0,1]$ with the norm $$\parallel f\parallel_{\mathcal{B}}=\sup\limits_{x\in(0,1]}|x^{\gamma}f(x)|,$$ for a fixed\footnote{In \eqref{a4'} and \eqref{a5} we need the assumptions to hold
only for a single $\gamma$.} $\gamma>0$. When equipped with the norm $\parallel \cdot\parallel_{\mathcal{B}}$, $\mathcal{B}$ is a Banach space. 

For $\Phi\in L^1$, let 
\begin{equation}\label{theF}
F_\eps(\Phi):=1_\Delta\Phi + (1-1_\Delta)\sum_{\omega\in\Omega}\Phi\circ g_{\omega,\eps}g_{\omega,\eps}'.
\end{equation}
Note that $F_\eps$ is a linear operator. In fact, for $x\in[0,1]\setminus \Delta$, the formula of $F_{\eps}$ can be re-written using the Perron-Frobenius operator of $T_\eps$: 
\begin{equation}\label{eq:renewal}
F_\eps(\Phi):=1_\Delta\Phi + (1-1_\Delta)\sum_{k\ge1}L^{k}_{\eps}(\Phi\cdot 1_{\{R_\eps>k\}}),
\end{equation}
where $L_\eps$ is the Perron-Frobenius operator associated with $T_\eps$; i.e., for $\varphi\in L^{\infty}$ and $\psi\in L^1$
$$\int\varphi\circ T_\eps\cdot\psi dm=\int \varphi\cdot L_\eps\psi dm.$$
It is given in \cite{BS} that the densities of the original system and the induced one are related (modulo normalization in the finite measure case) by
\begin{equation}\label{eq:density}
h_\eps = F_\eps(\hat h_\eps).
\end{equation}

We also define the following operator, which represents $\partial_\eps F_\eps\Phi|_{\eps=0}$ 
\begin{equation}\label{eq:Q}
Q\Phi = (1-1_\Delta) \sum_\omega \Phi'\circ g_\omega\cdot a_\omega g_\omega' + \Phi\circ g_\omega\cdot b_\omega,
\end{equation}
where $a_\omega=\partial_\eps g_{\omega,\eps}|_{\eps=0}$ and $b_\omega=\partial_\eps g_{\omega,\eps}'|_{\eps=0}$.

\subsection{Interval maps with countable number of branches}\label{ss:hatT}
We introduce here a class of interval maps which are uniformly expanding,
with a finite or countable number of branches, for which we will be able to 
prove a linear response formula. The induced map in Subsection \ref{ss:T} is a particular case of such uniformly expanding maps.\\

Let $\Delta$ be an interval and $V$ be a neighborhood of $0$. Let $\Omega$ be a finite or countable set.
We assume that the maps $\hat T_\eps\colon \Delta\to\Delta$ satisfy
\begin{itemize}
\item
For each $\eps\in V$, there exists a partition (mod 0) of $\Delta$ into open intervals $\Delta_{\omega,\eps}$, $\omega\in\Omega$ such that the restriction of $\hat T_\eps$ to $\Delta_{\omega,\eps}$ is piecewise $C^3$, onto and uniformly expanding in the sense that $\inf_\omega\inf_{\Delta_{\omega,\eps}}|\hat T_{\omega,\eps}'|>1$. We denote by $g_{\omega,\eps}$ the inverse branches of $\hat T_\eps$ on $\Delta_{\omega,\eps}$.
\item
We assume that for each $\omega\in\Omega$ and $j=0,1,2$ the following partial derivatives exist and satisfy the commutation relation\footnote{Note that \eqref{commB} is satisfied when $\hat T_\eps$
is an induced map as in Subsection~\ref{ss:T}. In particular, for each $i=0,1$ and $j=0,1,2$ the following partial derivatives exist and satisfy the commutation relation $\partial_\eps g_{i,\eps}^{(j)} = (\partial_\eps g_{i,\eps})^{(j)}$.}
\begin{equation}\label{commB}
\partial_\eps g_{\omega,\eps}^{(j)} = (\partial_\eps g_{\omega,\eps})^{(j)}.
\end{equation}
\item
We assume
\begin{equation}\label{a2}
\sup_\omega\sup_{\eps\in V}\sup_{x\in\Delta} \left|\frac{g_{\omega,\eps}''(x)}{g_{\omega,\eps}'(x)}\right|<\infty;
\end{equation}
and for $i=2,3$
\begin{equation}\label{a1}
\sum_{\omega}\sup_{\eps\in V}\sup_{x\in\Delta}|g^{(i)}_{\omega, \eps}(x)|<\infty;
\end{equation}
and for $i=1,2$
\begin{equation}\label{a3}
\sum_\omega\sup_{\eps\in V}\sup_{x\in\Delta}|\partial_\eps g_{\omega,\eps}^{(i)}(x)| <\infty. 
\end{equation}
\end{itemize}

Let $\hat L_\eps$ denote the transfer operator of the map $\hat T_\eps$; i.e., for  $\Phi\in L^1(\Delta)$ 
$$\hat L_\eps\Phi(x) :=\sum_{\omega\in\Omega}\Phi\circ g_{\omega,\eps}(x)g_{\omega,\eps}'(x)$$
for a.e. $x\in\Delta$. Under these conditions it is well known that $\hat T_\eps$ admits a unique (up to multiplication) finite absolutely continuous invariant measure. We denote its density by $\hat h_\eps$. Hence $\hat L_\eps\hat h_\eps=\hat h_\eps$. Moreover, $\hat L_{\eps}$ has a spectral gap when acting on $C^k$ and $W^{k,1}$
, $k=1,2$.
We denote the  Perron-Frobenius operator of the unperturbed map $\hat T$ by $\hat L$; i.e., $\hat L:=\hat L_0$ and let $\hat h:=\hat h_0$.   

\subsection{Linear response formula}

In \cite{BS} it is shown that the invariant density $\hat h_\eps$ of the induced map $\hat T_\eps$ is differentiable as a $C^0$ element and its linear response formula is given by
\begin{equation}\label{eq:response_ind}
\hat h^*:=(I-\hat L)^{-1}\hat L[A_0\hat h'+B_0\hat h],
\end{equation}
where $\hat h'$ is the spatial derivative of $\hat h$ and
$$A_0=-\left(\frac{\partial_{\eps}\hat T_{\eps}}{\hat T'_{\eps}}\right){\Big{|}}_{\eps=0},\hskip 0.5cm B_0= \left(\frac{\partial_{\eps}\hat T_{\eps}\cdot \hat T_{\eps}''}{\hat T_\eps'^2}-\frac{\partial_{\eps}\hat T_{\eps}'}{\hat T_\eps'}\right){\Big{|}}_{\eps=0}.$$ 
Moreover, for the original map, 
$\eps\mapsto h_\eps$ is differentiable as an element of $\mathcal B$; in particular, if the conditions hold for some $\gamma<1$
$$\lim_{\eps\to 0}||\frac{h_{\eps}-h}{\eps}-h^*||_{1}=0,$$
and $h^*$ is given by 
\footnote{Note that in the finite measure case, $h^*$ is the derivative of the non-normalized density $h_\eps$. The advantage in working with $h_{\eps}$ is reflected in keeping the operator $F_{\eps}$ linear and to accommodate the infinite measure preserving case. In the finite measure case, once the derivative of $h_{\eps}$ is obtained, the derivative of the normalized density can be easily computed. Indeed, $h_\eps=h +\eps h^*+o(\eps)$. Consequently, $\int h_\eps=\int h +\eps \int h^*+o(\eps)$. Hence, $\partial_{\eps}(\frac{h_\eps}{\int h_\eps}){|}_{\eps=0}=h^*-h\int h^*$.}

\begin{equation}\label{eq:lrf}
h^* = F_0 \hat h^*  + Q\hat h.
\end{equation}

\subsection{Main result and explicit strategy}
 We focus on the case $\gamma<1$. The goal of this work is to provide a numerical scheme that can rigorously approximate $h^*$, up to a pre-specified error $\tau>0$, in the $L^1$-norm. To obtain such a result we follow the following steps:
 \begin{enumerate}
 \item first provide a sequence of finite rank operators $\hat L_\eta$ that can be used to approximate the linear response for the induced map $\hat h^*$ in $L^1(\Delta)$. Since the formula of $\hat h^*$ involves $\hat h$ and $\hat h'$, we will design $\hat L_\eta$ so that its invariant density, $\hat h_\eta$, well approximates  $h_\eta$ in the $C^1$-norm,
 \item we pull-back to the original map by defining $F_0^{\text{app}}$ and $Q_0^{\text{app}}$ by truncating \eqref{theF} and \eqref{eq:Q}; i.e., for $\Phi\in L^1$,
 $$F_0^{\text{app}}(\Phi):=1_\Delta\Phi + (1-1_\Delta)\sum_{\omega=1}^{N^*}\Phi\circ g_{\omega,0}g_{\omega,0}'$$
 and 
 $$Q^{\text{app}}\Phi = (1-1_\Delta) \sum_{\omega=1}^{N^*} \Phi'\circ g_\omega\cdot a_\omega g_\omega' + \Phi\circ g_\omega\cdot b_\omega$$
\item finally, find $N^*$ large enough and set
\begin{equation}\label{eq:approxh^*}
h^*_\eta:=F_0^{\text{app}}\hat h_\eta^*+Q^{\text{app}}\hat h_\eta
\end{equation}
so that
\begin{equation*}
\| h^*_\eta-h^*\|_1\le \tau.
\end{equation*}
 \end{enumerate}

This strategy allows us to prove the following theorem.

\begin{theorem}\label{thm:main} For any $\tau>0$, there exists a sequence of finite rank operators $\hat L_\eta:L^{1}(\Delta)\to L^1(\Delta)$ such that for $\eta>0$ small enough and $N^*>0$ large enough 
$$\| h^*_\eta-h^*\|_1\le \tau.$$ 
\end{theorem}

\subsection{The validated numerics toolbox}
While the strategy for the approximation of the linear response may seems quite simple, to make it rigorous, i.e., with a certified 
control on the error terms so that the results have the strength of proofs, many 
different quantities have to be estimated explictly by means of a priori and 
a posteriori estimates.

The main toolbox we use for these validated estimates consists in 
\begin{enumerate}
\item Interval Arithmetics and rigorous contractors as the Interval Newton Method and the 
Shooting Method \cite{Tucker}
\item discretization of the transfer operator, using the Ulam and Chebyshev basis \cite{GN,GMNP,Wo}
\item a priori estimate on the tail of a series and rigorous bounds for a finite number of terms.
\end{enumerate}
We will introduce these methods and some of their implementation details during the proof of our result, showing the difference with the cited references when needed.

\section{Approximating the invariant density of the induced map}\label{sec:density}
To approximate the invariant density for the induced map two approximation steps are needed.
First we need to approximate the induced map, which has countable branches with a
map with a finite number of branches. Then, we will discretize the transfer operator 
of this map by using a Chebyshev approximation scheme.

\subsection{From countable branches to finite branches}\label{ss:ApproxMethod}
Let $\delta_k>0$ with $\delta_k=\left | \cup_{n=k}^{\infty}[\omega]\right|$. To simplify notation we assume without loss of generality that $\frac12\in \overline{\cup_{n=k}^{\infty}[\omega]}$. Let

$$\hat T_{\delta_k}(x)=\begin{cases}
       \hat T(x)     & \mbox{ , if } x\in[\delta_k,1],\\
       \frac12\delta_k^{-1}(x-\frac12) +\frac12& \mbox{ , if } x\in[0.5,\delta_k).\\
       \end{cases}$$
Then the transfer operator $\hat L_{\delta_k}$, associated with $T_{\delta_k}$ is 
acting on $\Phi\in L^1(\Delta)$ as: 
$$\hat L_{\delta_k}\Phi(x) :=\sum_{\overset{\omega\in\Omega}{n<k}}\Phi\circ g_{\omega,\eps}(x)g_{\omega,\eps}'(x)+\Phi\left(\delta_k(2x-1)+\frac{1}{2}\right)2\delta_k$$
for a.e. $x\in\Delta$.
\begin{lemma}\label{lem:tripleC2}
Let $\Phi\in C^1$, then 
$$\|(\hat L-\hat L_{\delta_k})\Phi\|_{C^1}\le(D+D_0D+2)\|\Phi\|_{C^1}\delta_k,$$ 
where $D_0=\Big\|\frac{g_{\omega}''}{g'_\omega}\Big\|_\infty$ and $D\geq 2\sup_{\omega}\frac{|g'_\omega(x)|}{|g'_\omega(y)|}$ for all\footnote{The existence of a uniform constant $D>0$ is implied by condition \eqref{a2}.} $x,y$ in $[0.5,1]$.
\end{lemma}
\begin{proof}
First notice that
\begin{equation}\label{eq:est1}
\begin{split}
\left|(\hat L-\hat L_{\delta_k})\Phi\right|&= \left|\sum_{\overset{\omega\in\Omega}{n\ge k}}\Phi\circ g_{\omega}(x)g_{\omega}'(x)-\Phi(\delta_k(2x-1)+\frac12)2\delta_k\right|\\
&\le \sum_{\overset{\omega\in\Omega}{n\ge k}}|\Phi\circ g_{\omega}(x)|\cdot|g_{\omega}'(x)|+|\Phi(\delta_k(2x-1)+\frac12)|2\delta_k
\end{split}
\end{equation}
and
\begin{equation}\label{eq:est2}
\begin{split}
&\left|\left((\hat L-\hat L_{\delta_k})\Phi\right)'\right|=\\ 
&\left|\sum_{\overset{\omega\in\Omega}{n\ge k}}\left(\Phi'\circ g_{\omega}(x)(g_{\omega}'(x))^2+\Phi\circ g_{\omega}(x)g_{\omega}''(x)\right)\right. 
\left. -|\Phi'(\delta_k(2x-1)+\frac12)4\delta_k^2|\right|\\
&\le\sum_{\overset{\omega\in\Omega}{n\ge k}}|\Phi'\circ g_{\omega}(x)|\cdot (g_{\omega}'(x))^2+ \sup_\omega\Big\|\frac{g_{\omega}''}{g'_\omega}\Big\|_\infty\sum_{\overset{\omega\in\Omega}{n\ge k}}|\Phi\circ g_{\omega}(x)|\cdot|g_\omega'(x)|\\
&+|\Phi'(\delta_k(2x-1)+\frac12)|4\delta_k^2.
\end{split}
\end{equation}
Now notice that by the Mean Value Theorem, $\exists\, \xi_\omega\in(\frac12,1)$ such that
$$|g_\omega(1)-g_\omega(\frac12)|=|g'_\omega(\xi_\omega)|/2.$$
Therefore, 
\begin{equation}\label{eq:est3}
|g'_\omega(x)|\le 2 
|g_\omega(1)-g_\omega(\frac12)|\cdot \sup_\omega\frac{|g'_\omega (x)|}{|g'_\omega(\xi_\omega)|}:= D\cdot |g_\omega(1)-g_\omega(\frac12)|.
\end{equation}
Thus, using \eqref{eq:est3} in \eqref{eq:est1} and \eqref{eq:est2}, we obtain

\begin{equation}\label{eq:est4}
\begin{split}
&\left|(\hat L-\hat L_{\delta_k})\Phi\right|+\left|\left((\hat L-\hat L_{\delta_k})\Phi\right)'\right|\le\|\Phi\|_{C^0}D\sum_{\overset{\omega\in\Omega}{n\ge k}}|g_\omega(1)-g_\omega(\frac12)|+\|\Phi\|_{C^0}2\delta_k\\
&\hskip 0.5cm+\left(\|\Phi'\|_{C^0}D+\|\Phi\|_{C^0}D\sup_\omega\Big\|\frac{g_{\omega}''}{g'_\omega}\Big\|_\infty\right)\sum_{\overset{\omega\in\Omega}{n\ge k}}\cdot |g_\omega(1)-g_\omega(\frac12)|+\|\Phi'\|_{C^0}4\delta_k^2\\
&\hskip 0.5cm =(D+D_0D+2)\|\Phi\|_{C^0}\delta_k+\|\Phi'\|_{C^0}4\delta_k^2\le(D+D_0D+2)\|\Phi\|_{C^1}\delta_k.
\end{split}
\end{equation}
\end{proof}
The next lemma shows that using the above information, the densities $\hat h$ and $\hat h_{\delta_k}$ can be made arbitrarily close in $C^1$.

\begin{lemma}\label{lem:GeneralisedFixedPointLemma}
For two operators, $L_1$ and $L_2$, with fixed points $h_1$ and $h_2$ normalised with respect to $||\cdot||_{1}$, and a shared Lasota-Yorke innequality
\begin{equation*}
\norm{L_i^nf}_{s}\leq A\lambda^n\norm{f}_{s}+B\norm{f}_{w}
\end{equation*}
for $i\in \{1,2\}$, there is a $C^*$ such that $\norm{L_i^nf}_{s}\leq C^*\norm{f}_{s}$,
and for any $N\ge 1$ we have
\begin{equation*}
||h_1-h_2||_{s}\leq ||L_1^{N}(h_1-h_2)||_{s}+NC^*||(L_1-L_2) h_2||_{s}.
\end{equation*}
Furthermore if $\norm{L_1^N|_{U^0}}_{s}\leq C_N<1$ then we can have
\begin{equation*}
||h_1-h_2||_{s}\leq \frac{NC^*||(L_1-L_2) h_2||_{s}}{1-C_N}.
\end{equation*}
\end{lemma}

\begin{proof}
The value of $C^*$ is given by $A\lambda+B$, and the distance between the two fixed points is shown as follows,
\begin{eqnarray*}
||h_1-h_2||_{s} &\leq &||L_1^{N}h_1-L_2^{N}h_2||_{s} \\
&\leq &||L_1^{N}(h_1-h_2)||_{s}+||(L_1^{N}-L_2^{N})h_2||_{s}.
\end{eqnarray*}%
Note that
\begin{eqnarray*}
(L_1^{N}-L_2^{N})h_2
&=&\sum_{k=1}^{N}L_1^{N-k}(L_1-L_2)L_2^{k-1}h_2\\
&=&\sum_{k=1}^{N}L_1^{N-k}(L_1-L_2)h_2.
\end{eqnarray*}%
Consequently, 
\begin{eqnarray*}
||(L_1^{N}-L_2^{N}) h_2||_{s} &\leq
&\sum_{k=1}^{N}C^*||(L_1-L_2)h_2||_{s} \\
&\leq &NC^*||(L_1-L_2)h_2||_{s}.
\end{eqnarray*}
Given that $||h_1||_1=||h_2||_1$ we have that $h_1-h_2\in U^0$ and therefore we can bound $||L_1^{N}(h_1-h_2)||_{s}$ by $C_N||h_1-h_2||_{s}$, rearranging gives us the last result.
\end{proof}

\begin{remark}\label{rem:approx1}
The operators $\hat L$ and $\hat L_{\delta_k}$ admit a uniform Lasota-Yorke inequality,
\begin{equation*}
\norm{\hat L^nf}_{C^1}\leq A\lambda^n\norm{f}_{C^1}+B\norm{f}_{C^0}
\end{equation*}
as shown in section \ref{a:Cstar}, where a value for $C^*$ is found. Bounds on $\norm{L_{\delta_k}|_{U^0}}_{C^1}$ can be found by techniques described in section \ref{NR:Convergence}.
The $C^{1}$ norms of $\hat h_{\delta_k}$ and $\hat h$ can be estimated using the Lasota-Yorke inequalities. We can then use lemma \ref{lem:tripleC2} to make the error in $||\hat h_{\delta_k}-\hat h||_{C^{1}}$ as small as we like.
\end{remark}
Next we define a finite rank operator to obtain $\hat h_n$ so that $\|\hat h_n-\hat h_{\delta_k}\|_{C^1}$ can be made as small as required.

\subsection{Approximating the invariant density for $\hat{T}_{\delta_k}$}
To approximate the invariant density, we will discretize the operator $\hat{L}_{\delta_k}$
using the basis of the Chebyshev polynomials of the first kind.
The Chebyshev basis is a basis for the space of polynomials with a main advantage: given
a continuous function $f$ on $[-1,1]$ the interpolating polynomial on the Chebyshev points
are ``near-best'' approximants with respect to $||.||_{\infty}$ \cite[Theorem 16.1]{Tr}; moreover if the function $f$ is regular enough the coefficients of the interpolant decay ``fast'' and
are easily computed by means of the Fast Fourier Transform. 

Before going forward, some observations are in order, since Chebyshev polynomials do not solve all the problems involved with approximation: to apply this approximation scheme we need to prove a priori that our stationary density is regular enough and keep track of all the errors involved in the computation of the coefficients. Moreover, evaluating a Chebyshev polynomial of high degree rigorously is a delicate matter \cite{LedouxMoroz}.

\subsubsection{Chebyshev interpolation and projection}
The material for this section comes from \cite{Tr}, \cite{Hi} and \cite{XCW}. 

Given an $f\in W^{k,p}$ from $[-1,1]\to \mathbb{C}$, we can define a 
function $\mathcal{F}(\theta)$ on $[0, 2\pi]$ by
\begin{equation*}
\mathcal{F}(\theta)=f(\cos(\theta)).
\end{equation*}
The Fast Fourier Transform (FFT) on a grid of size $2N$ allows us to compute the coefficients $a_k$ of the trigonometric polynomial interpolating $\mathcal{F}$ on an equispaced grid $\theta_i = (2\pi i)/(2N)$, for $i$ in $0,\ldots, 2N-1$.

Let $x_i = \cos(\theta_i)$ for $i$ in $1, \ldots, N$; 
observe that $\mathcal{F}(\theta_i) = \mathcal{F}(\theta_{2N-i}) = f(x_i)$. 
We define the Chebyshev polynomials by the relation
\[
T_n(\cos(\theta))=\cos(n\theta), 
\]
then, if we denote by $b_0 = a_0/2$, $b_{N-1} = a_{N-1}/2$ and $b_i = a_i$ for 
all $i=1, \ldots, N-2$:
\[
p(x) = \sum_{k=0}^{N-1} b_k T_k(x),
\]
where the $a_k$ are the ones computed by the FFT is the interpolating polynomial of $f$ on the grid given by the $x_i$.

\begin{definition}
Let $f\in W^{k,1}$, $k>1$, we define the \textbf{Chebyshev (interpolating) projection}
\begin{equation*}
\pi_n f = \sum_{k=0}^n a_k T_k(x).
\end{equation*}
\end{definition}

\begin{remark}
If, instead of the FFT we had taken the Fourier transform of $\mathcal{F}$, the Fourier coefficients $\hat{a}_k$ would define coefficients $\hat{b}_k$, the Chebyshev orthogonal expansion
\[
f(x) = \sum_{k=0}^{+\infty}\hat{b}_k T_k(x),
\]
and the \textbf{Chebyshev projection}
\[
\hat{\pi}_n f = \sum_{k=0}^n \hat{b}_k T_k(x).
\]
The coefficients $a_k$ and $\hat{a}_k$ are related by the aliasing relation:
\[
   a_k = \sum_{p\in \mathbb{Z}} \hat{a}_{k+p 2n},
\] 
a-priori knowledge of the regularity of $f$ allows to estimate the aliasing error above.
\end{remark}

This foundational Theorem from \cite{Tr} estimates the decay rate of the Chebyshev coefficients.
\begin{theorem}\label{th:Tr}
For an integer $\nu\geq0$, let $f$ and its derivatives through $f^{(\nu-1)}$ be absolutely continuous on $[-1,1]$ and suppose the $\nu^{th}$ derivative $f^{(\nu)}$ is of bounded variation $V$. Then for $k\geq\nu+1$, the Chebyshev coefficients of $f$ satisfy
\begin{equation*}
|\hat{b}_k|\leq \frac{2V}{\pi k(k-1)\dots (k-\nu)}\leq \frac{2V}{\pi(k-\nu)^{\nu+1}}.
\end{equation*}
\end{theorem}

The decay rate of Chebyshev coefficients allows us to estimate the projection error in $C^0$ and
$C^1$ norm.
\begin{theorem}\label{thm:approx_error_C0}
If $f$ satisfies the conditions of Theorem \ref{th:Tr}, with $V$ again the total variation of $f^{(\nu)}$ for some $\nu\geq 1$, then for any $n\geq \nu$, its Chebyshev projection satisfies
\begin{equation*}
\norm{f-\pi_n f}_\infty\leq \frac{2V}{\pi \nu n(n-1)\dots(n+1-\nu)}
\end{equation*}
\end{theorem}
The following theorem is a consequence of combining theorem \ref{th:Tr}, with the proof of theorem 2.3 from \cite{XCW},
\begin{theorem}\label{thm:approx_error_C1}
If $f,f',\dots,f^{(\nu-1)}$ are absolutely continuous on $[-1,1]$ and if $\norm{f^{(\nu)}}_1=V<\infty$ for some $\nu\geq 0$, then for each $n\geq \nu+1$, we have that for $\nu>2$
\begin{equation*}
\norm{f'-(\pi_n f)'}_\infty\leq \frac{4(n+1)V}{n(\nu-2)\pi (n-2)(n-3)\dots (n+1-\nu)}.
\end{equation*}
\end{theorem}
\begin{proof}
From the proof of theorem 2.3 from \cite{XCW} we have
\begin{equation*}
\norm{f'-(\pi_n f)'}_\infty\leq 2\sum_{k=n+1}^\infty |a_k|k^2
\end{equation*}
and theorem \ref{th:Tr} then gives
\begin{align*}
\norm{f'-(\pi_n f)'}_\infty\leq &\sum_{k=n+1}^\infty \frac{4Vk^2}{\pi k(k-1)\dots (k-\nu)}\\
\leq &\frac{4(n+1)V}{n(k-2)\pi (n-2)(n-3)\dots (n+1-k)}.
\end{align*}
\end{proof}
We can use these theorems to bound the error of Chebyshev projections in the $C_1$ norm.

\subsubsection{Numerical remarks: FFT and Chebyshev}


It is important to have an explicit estimate of the error on the coefficients introduced by 
the FFT.
The main issue here is that, when we computing Chebyshev points and evaluating the function $f$
are not exact operation. To compute rigorous inclusions of the true mathematical value, we use Interval Arithmetics \cite{Tucker}.

This means that we need to compute the FFT of a vector of intervals, not of floating point numbers.
The following is the consequence of a classical result from \cite{Hi} that allows us
to find a vector of intervals that encloses the Fast Fourier Transform of any element
of the vector of the values. This allows us to use optimized implementations of the FFT
algorithm as FFTW \cite{FJ}.

\begin{lemma}
Let $\tilde{f}$ be a vector of intervals of dimension N, $f_m$ the vector of their midpoints, $f_r$ the vector of their radiuses. Let $\hat a$ be the computed FFT of $f_m$. Then
\begin{equation*}
\norm{\hat a-\bar a}_2\leq \frac{\log_2(N)}{\sqrt{N}}(\frac{\eta}{1-\eta}\norm{f_m}_2+\norm{f_r}_2).
\end{equation*}
where $\bar{a}$ 
is the exact FFT for any $f\in \tilde{f}$, $\eta=\mu+\gamma_4(\sqrt{2}+\mu)$ with $\mu$ the absolute error in the computation of the twiddle factors and $\gamma_4=4u/(1-4u)$ where $u$ is the unit roundoff.
\end{lemma}

\subsubsection{Bounding the error on the invariant density}
Let $\pi_n$ be the Chebyshev projection and let
\begin{equation*}
\hat L_\eta=\pi_n \hat L_{\delta_k} \pi_n
\end{equation*}
be our finite rank approximation of $\hat L_{\delta_k}$.
\begin{lemma}\label{lem:WLY}
If $1/\hat{T}_{\delta_k}$ is of class $C^\nu$ then $\hat L_{\delta_k}$ admits Lasota-Yorke like inequalities of the form
\begin{equation*}
\norm{(\hat L_{\delta_k}^n f)^{(k)}}_1\leq (\lambda^k)^n\norm{f^{(k)}}_1+ \tilde A_k \norm{f}_{W^{k-1,1}},
\end{equation*} 
for some $\nu\in\mathbb{N}$ for $k=1,\dots,\nu$.
This implies that if $h_{\delta_k}$ is a fixed point of $\hat{L}_{\delta_k}$
\[
||h_{\delta_k}||_{W^{k, 1}}\leq \frac{A_{k}}{1-\lambda^k}||h_{\delta_k}||_{W^{k-1, 1}}   
\]
\end{lemma}
\begin{remark}
The Lasota-Yorke inequalities give us an upper bound on the $W^{k,1}$ norm of the fixed point.
This, together with Theorems \ref{thm:approx_error_C0} and \ref{thm:approx_error_C1} permit us 
to control the discretization error.
To estimate our error, we need to compute the constants of this Lasota-Yorke inequality explicitly, we refer to Subsection \ref{ss:autoly} for the tecnique we used.
\end{remark}

\begin{remark}\label{rem:hetabound}
We can use theorems \ref{thm:approx_error_C0} and \ref{thm:approx_error_C1}, together with lemma \ref{lem:WLY} to get a bound on $||\hat L_{\delta_k}-\hat L_n||_{C^1}$ and the techniques from section \ref{NR:Convergence} to with lemma \ref{lem:GeneralisedFixedPointLemma} using operators $\hat L_{\delta_k}$ and $\hat L_n$ in order to bound $\norm{\hat h_{\delta_k}-\hat h_n}_{C^1}$ explicitly.

This approach is now quite established, a full treatment can be found in \cite{GN, GMNP}. 
\end{remark}
\begin{remark}\label{rem:integralinjection}
The discretized operator obtained by the Chebyshev discretization does not preserve the
value of the integral.
To solve this issue, as in \cite{GMNP} we correct the behaviour of the discretized operator 
by defining a new operator
\[
\hat{Q}_n f = \hat{L}_{n}f + 1 \cdot (\int f dx- \int \hat{L}_{n}f dx)   
\]
which is guaranteed to preserve the space of average $0$ measure and has eigenvalue $1$,
since the row vector that contains the integrals of the basis elements is preserved by multiplication on the left.
\end{remark}

\subsubsection{Numerical remarks: convergence rates}\label{NR:Convergence}
The problem of bounding the error in $C^1$ is now reduced to estimate 
$c_N$ such that $\norm{\hat L_{n}^N|_{U^0}}_{C^1}\leq c_N$.
Since the operator $L_{n}$ is of finite rank, we can use numerical methods to 
compute these quantities in a rigorous way.

Given a basis $\{e_i\}_i$ of $U^0$ normalized with respect to the $C^1$ norm,
a generic function $v$ in $U^0$ is written as $v=\sum_{i=1}^N b_i e_i$. We want to find a
constants $C_k$ such that
\begin{equation*}
||\hat L_{n}^k v||_{C^1}\leq C_k||v||_{C^1}.
\end{equation*}
If a $c_k$ exists such that for all basis elements $e_i$ we have
\begin{equation*}
||\hat L_{n}^k e_i||_{C^1}\leq c_k.
\end{equation*}
then 
\begin{equation*}
||\hat L_{n}^k v||_{C^1}\leq c_k\sum_{i=1}^N|b_i|\leq c_k ||b||_{\ell^1}
\end{equation*}
where the $\ell_1$ norm is the linear algebra norm on the coefficients; we will exhibit a constant $D$ such that $||.||_{\ell^1}\leq D||.||_{C^1}$;
then $C_k \leq Dc_k$.

A basis of $U^0$ in the Chebyshev basis  is given by
\begin{align*}
e_i=&\frac{g_i}{\norm{g_i}_{C^1}}, \quad\textrm{where}\quad g_i(x)=T_i(x)-\int_{-1}^1T_i(x)dx.
\end{align*}
We can link the Chebyshev coefficients $a_i$ and $b_i$ by
\begin{equation*}
a_0=-\sum_{i=1}^{N}b_i\int_{-1}^1T_i(x)dx
\end{equation*}
and $a_i=b_i$ for $i>0$.
We can use theorem \ref{th:Tr} to say
\begin{equation*}
|a_k|\leq \frac{\norm{v}_{C^1}}{k}
\end{equation*}
and therefore
\begin{equation*}
||b||_{\ell^1} = \sum_{k=1}^N|b_i|\leq \sum_{k=1}^N\frac{\norm{v}_{C^1}}{k}\leq \log(N+1)\norm{v}_{C^1}
\end{equation*}
so we have $D=\log(N+1)$ and
\begin{equation*}
\norm{\hat L_{n}^k v}_{C^1}\leq c_k\log(N+1)\norm{v}_{C^1}.
\end{equation*}

Computationally if we take $N$ functions
\begin{equation*}
\hat e_i=\frac{g_i}{\norm{g_i}_{C^1}^-},
\end{equation*}
where $\norm{g_i}_{C^1}^-$ is a lower bound on the $C^1$ norm of $g_i$ and if we calculate each
\begin{equation*}
\norm{\hat L_{n}^k e_i}_{C^1}
\end{equation*}
then take the maximum value and call it $\hat c_k$, then $\hat c_k\log(N+1)$ is an upper bound on the $C^1$ contraction of $\norm{\hat L_{n}^k|_{U^0}}_{C^1}$.

It is important to explain how we compute an upper bound for the $C^1$ norm: we use a classical
optimization algorithm in IntervalArithmetics \cite{Tucker} that allows us to give a certified upper
bound, implemented in the Julia package IntervalOptimisation.jl.
The main issue here is that the Clenshaw algorithm is prone to overestimation when evaluated on intervals \cite{LedouxMoroz}
; to solve this we extended the algorithms in 
\cite{LedouxMoroz} to get tighter bound for the maximum of a Chebyshev polynomial and its
derivative.

\section{Approximating the linear response for the induced system}\label{sec:InducedResponse}
We now provide an approximation of $\hat h^*$ in the $C^0$-norm, through the use of the hat
approximation; we refer to \cite{GMNP,GN} for an in deep treatment of the hat discretization.
\begin{definition}\label{def:Hat}
The hat projection, is a projection $\Pi_h: C^0([0.5,1])\to C^0([0.5,1])$;
let $\{x_i\}$ be $h:=1/\eta+1$ equispaced points in $[0.5, 1]$, such that
$|x_i-x_{i-1}|=\eta$.

We define  
\[
\phi_i(x) = \left\{
\begin{array}{cc} 
\frac{|x-x_i|}{\eta} & \textrm{if }|x-x_i|\leq \eta \\
0 & \textrm{if }|x-x_i|> \eta
\end{array}\right.   
\]
\begin{equation*}
\Pi_h f(x)= \sum_{i=0}^h f(x_i)\phi_i(x).
\end{equation*}
The Hat discretization of the transfer operator $\hat L_{\delta_k}$ is defined by
\[
\hat L_h :=\Pi_h\hat L_{\delta_k}\Pi_h.
\]

In a similar fashion as in Remark \ref{rem:integralinjection}, this operator does not preserve the integral
so an auxiliary operator is defined, $\hat P_h:C^0([0.5, 1])\to C^0([0.5, 1])$
\[
\hat P_h f = \hat L_w f + 1\cdot \left(\int \hat L_w f - \int \hat L f\right)
\] 
where $1$ in the equation above is the constant function equal to $1$ on $[0.5, 1]$.

Since $L_{\delta_k}, \hat L_w$ and $\hat P_h$ satisfy a uniform Lasota-Yorke inequality
with respect to the Lipschitz seminorm and the $C^0$ norm and 
\[
||\pi f-f||_{\infty}\leq \textrm{Lip(f)}\eta   
\]
there exists a constant $K$ such that 
\[
||\hat P_h f-\hat L_{\delta_k}f||_{\infty}\leq K (||f||_{C^0}+\textrm{Lip}(f))\eta;  
\]
we refer again to \cite{GMNP, GN}.
\end{definition}

\subsection{Error in approximating the linear response}   
In order to approximate the linear response in the induced map we will need to use an approximation.

\begin{definition}
We define 
\[
W(x)=\hat L [A_0\hat h'+B_0\hat h](x).
\]
\end{definition}
\begin{definition}
Let
\begin{equation*}
   W_h(x)=a_0\phi_0(x)+\sum_{i=0}^h\frac{a_{i-1}+a_i}{2}\phi_i(x)+a_n\phi_n(x),
\end{equation*}
   where $\phi_i(x)$ are the hat functions in the definition of \ref{def:Hat} and
   \begin{align*}
   a_i =&\sum_{\omega\leq k} \frac{1}{\eta}\int_{g_\omega(I_i)}A_0(z)\cdot \hat h_n'(\zeta)+B_0(z)\hat h_n(\zeta)dz+\frac{1}{\eta}\int_{\delta_k(2I_i-1)+1/2}[A_0\hat h'_\eta+B_0\hat h_n](z)dz,
   \end{align*}
   with $\zeta$ chosen arbitrarily in $g_{\omega}(I_i)$ for each $\omega, I_i$.
\end{definition}
\begin{remark}\label{rem:WW_h}
It is possible to bound
\[
\left|a_i-\sum_{\omega}\frac{1}{\eta} \int_{g_\omega(I_i)}A_0(z)\cdot \hat h_n'(z)+B_0(z)\hat h_n(z)dz \right|= \left|a_i-\frac{1}{\eta}\int_{I_i} \hat L(A_0 h'+B_0 h) \right|
\]
as a decreasing function of $k$ and $\eta$, as done in subsection \ref{a:estimateshstar3}.
Let 
\[
\tilde{a}_i = \frac{1}{\eta}\int_{I_i} \hat L(A_0 h'+B_0 h);
\]
and 
\[
W_{\infty}(x) = \sum_{i=0}^h \tilde{a}_i \phi_i(x).
\]
Then 
\begin{align*}
||W_h-W||_{C^0}&\leq ||W_h-W_{\infty}||_{C^0}+||W_{\infty}-\Pi_w W||_{C^0}+||\Pi_w W-W||_{C^0}\\
&\leq ||W_h-W_{\infty}||_{C^0}+2\textrm{Lip}(W)\eta.
\end{align*}
This implies that the $C^0$ distance between $W_h$ and $W$ can be made as small as we 
want.
\end{remark}
   
\begin{definition}
Let 
\begin{equation}\label{eq:inducedresponseapprox}
\hat h_\eta^*:=\sum_{n=0}^{l^*}\hat Q_{w}^{n} W_h
\end{equation}
where $\hat P_h$ is the hat approximation of $\hat L_{\delta_k}$.
\end{definition}

%

\begin{lemma}\label{thm:induced_resp}
Recall from \eqref{eq:response_ind} that
\begin{equation*}
\hat h^* =(I-\hat L)^{-1}\hat L[A_0\hat h'+B_0\hat h]
\end{equation*}
We have
\begin{equation}
\begin{split}
\|\hat h_\eta^*-\hat h^*\|_{C^0}\leq&\sum_{n=0}^{l^*}\sum_{i=0}^{n}\norm{(\hat L-\hat L_{\delta_k})\hat P_h^{n-i}W_h}_{C^0}+\sum_{n=0}^{l^*}\sum_{i=0}^{n}\norm{(\hat L_{\delta_k}-\hat Q_{w})\hat P_h^{n-i}W_h}_{C^0}\\
+&l^*\norm{(W_h-W)}_{C^0}+\sum_{n=l^*+1}^{\infty}\norm{\hat L^n W}_{C^0}.
\end{split}
\end{equation}
\end{lemma}
\begin{proof}
First of all we observe that  by \eqref{eq:response_ind} and the definition of $\hat h_{\eta}^*$ we have
\begin{align*}
\norm{\hat h^*-\hat h_\eta^*}_{C^0}=&\norm{\sum_{n=0}^\infty \hat L^n W-\sum_{n=0}^{l^*}\hat P_h^n W_h}_{C^0}\\
\leq&\norm{\sum_{n=0}^{l^*}\hat L^n W-\sum_{n=0}^{l^*}\hat P_h^n W_h}_{C^0}+\sum_{n=l^*+1}^{\infty}\norm{\hat L^n W}_{C^0}
\end{align*}

The estimate follows by direct calculation. Indeed, 
using the triangle inequality and the second resolvent identity, we have that
\begin{align*}
&\norm{\sum_{n=0}^{l^*}\hat L^n W_h-\sum_{n=0}^{l^*}\hat P_h^n W_h}_{C^0}+\norm{\sum_{n=0}^{l^*}(\hat L^n W-\hat L^n W_h)}_{C^0}+\sum_{n=l^*+1}^{\infty}\norm{\hat L^n W}_{C^0}\\
\leq& \sum_{n=0}^{l^*}\sum_{i=0}^{n}\norm{\hat L^i(\hat P_h-\hat L)\hat P_h^{n-i}W_h}_{C^0}+\sum_{n=0}^{l^*}\norm{\hat L^n(W_h-W)}_{C^0}+\sum_{n=l^*+1}^{\infty}\norm{\hat L^n W}_{C^0}\\
\leq& \sum_{n=0}^{l^*}\sum_{i=0}^{n}\norm{\hat L^i|_{U^0}}_{C^0}\norm{(\hat P_h-\hat L_{\delta_k})\hat P_h^{n-i}W_h}_{C^0}+\sum_{n=0}^{l^*}\sum_{i=0}^{n}\norm{\hat L^i|_{U^0}}_{C^0}\norm{(\hat L_{\delta_k}-\hat L)\hat P_h^{n-i}W_h}_{C^0}\\
&+\sum_{n=0}^{l^*}\norm{\hat L^n|_{U^0}}_{C^0}\norm{(W_h-W)}_{C^0}+\sum_{n=l^*+1}^{\infty}\norm{\hat L^n W}_{C^0}
\end{align*}
\end{proof}
\begin{remark} The estimates in Lemma \ref{thm:induced_resp} can all be made as small as desired. Indeed, notice that $W$ is a zero average $C^1$ function; therefore
\begin{itemize}
\item the Lipschitz and $C^0$ norm of $P_h W_h$ can be estimated explictly,
\item the last summand $\sum_{n=l^*+1}^{\infty}\|\hat L^{n+1}W||_{C^0}$ can be made small, for sufficiently large $l^*$, since $\hat L$ admits a spectral gap when acting on $C^0$. Once this term is estimated, $l^*$ is fixed once and for all;
\item the summand $\sum_{n=0}^{l^*}\sum_{i=1}^{n}\|(\hat L-\hat L_{\delta_k})\hat P_h^{n-i}W_h\|_{C^0}$ can be made small by choosing $\delta_k$ small enough;
\item the summand $\sum_{n=0}^{l^*}\sum_{i=1}^{n}\|(\hat L_{\delta_k}-\hat P_h)\hat P_h^{n-i}W_h\|_{C^0}$ can be made small by choosing $\eta$, the size of the hat discretization, small enough;
\item the term $\norm{\hat L^n|_{U^0}}_{C^0}\norm{(W_h-W)}_{C^0}$ can be made small by reducing $\delta_k$ and $\eta$.
\end{itemize}
\end{remark}
\section{Normalising the density and the linear response.}\label{sec:normalizing}
Ultimately the goal is to approximate the dynamics of the system, so we would like the invariant measure to be a probability measure. This is not always possible for maps with indeterminate fixed points, however it was shown in \cite{BT} that a fixed point of the transfer operator of an LSV map is bounded above by $Cx^{-\alpha}$, for some constant $C$, giving a maximum integral of $\frac{C}{1-\alpha}$, so we can make our calculated density a probability density by normalising with respect to its integral, which will give us a new error of
\begin{align*}
\norm{\frac{h_n}{\int h_n dm}-\frac{h}{\int h dm}}_1&=\norm{\frac{h_n}{\int h_n dm}-\frac{h}{\int h_\eta dm}+h\frac{\int h dm-\int h_n dm}{\int h dm \int h_n dm}}_1\\
\leq&\frac{\norm{h-h_n}_1}{\int h_n dm}+\frac{\norm{h}_1}{\int h dm}\frac{\norm{h-h_n}_1}{\int h_n dm}\leq2\frac{\norm{h-h_n}_1}{\int h_n dm}
\end{align*}
and if we ensure that the integral is preserved throughout the approximation then the error is $\frac{\norm{h^*-h^*_\eta}_1}{\int h^*_\eta dm}$. 

Since the Chebyshev approximation does not preserve the integral we use the first estimate to 
bound 
\[
\norm{\frac{h_n}{\int h_n dm}-\frac{h}{\int h dm}}_1
\]
where we calculate the integral of $h_n$ by
\begin{equation*}
\int_0^1F^{app}\hat h_ndx = \int_{0.5}^1\hat h_ndx+\sum_{|\omega|=1}^{N^*}H_n\circ g_\omega(1)-H_n\circ g_\omega(0.5)
\end{equation*}
where $H_n(x)=\int_{0.5}^xh_ndx$.

The linear response for the normalised invariant density is then $\rho^*$ such that
\begin{equation*}
 \lim_{\epsilon\to 0}\norm{\frac{\frac{h}{\int h dm}-\frac{h_\epsilon}{\int h_\epsilon dm}}{\epsilon}-\rho^*}_1 = 0
\end{equation*}
we get from this
\begin{align*}
\norm{\frac{\frac{h}{\int h dm}-\frac{h_\epsilon}{\int h_\epsilon dm}}{\epsilon}-\rho^*}_1=&\norm{\frac{\frac{h}{\int h dm}-\frac{h_\epsilon}{\int h+\epsilon h^*+o(\epsilon^2) dm}}{\epsilon}-\rho^*}_1\\
=&\norm{\frac{\frac{h}{\int h dm}-\frac{h_\epsilon}{\int h dm}}{\epsilon}-\frac{h_\epsilon\int [h^*+o(\epsilon)] dm}{\int h dm (\int [h+\epsilon h^*+o(\epsilon^2)] dm)}-\rho^*}_1
\end{align*}
which tells us that
\begin{equation*}
\rho^*=\frac{h^*}{\int h dm}-\frac{h\int h^* dm}{(\int h dm)^2}.
\end{equation*}

Letting 
\[A=\frac{h^*}{\int h dm}-\frac{h^*_\eta}{\int h_\eta dm}\quad\,  B=-\frac{h\int h^* dm}{(\int h dm)^2}+\frac{h_n\int h^*_\eta dm}{(\int h_n dm)^2}\] 
the error on the normalised linear response is calculated as follows,
\begin{align*}
&\norm{\frac{h^*}{\int h dm}-\frac{h\int h^* dm}{(\int h dm)^2}-\frac{h^*_\eta}{\int h_n dm}+\frac{h_n\int h^*_\eta dm}{(\int h_n dm)^2}}_1\\
=&\norm{\frac{h^*}{\int h_n dm}-\frac{h^*_\eta}{\int h_n dm}+\frac{h^*\int (h_n-h)dm}{\int h dm \int h_n dm}+B}_1\\
=& \norm{\frac{h^*-h^*_\eta}{\int h_n dm}+\frac{h^*\int (h_n-h)dm}{\int h dm \int h_n dm}+B}_1\\
=&\norm{A+\frac{h_n\int h^*_\eta dm}{(\int h_n dm)^2}-\frac{h\int h^* dm}{(\int h dm)^2}}_1,
\end{align*}
the expression above can be bounded from above by
\begin{align*}
&\norm{A+\frac{h_n\int h^*_\eta dm}{(\int h_n dm)^2}-\frac{(h_n+(h-h_n))(\int h^*_\eta dm-\norm{h^*-h^*_\eta}_1)}{(\int h_n dm+\norm{h-h_n}_1)^2}}_1\\
=&\left\lVert A-\frac{(h-h_n)(\int h^*_\eta dm-\norm{h^*-h^*_\eta}_1)-h_n\norm{h^*-h^*_\eta}_1}{(\int h_n dm+\norm{h-h_n}_1)^2}\right.\\
&+\left.\frac{h_n \int h_\eta^*dm(2\int h_n dm \norm{h-h_n}_1+ \norm{h-h_n}_1^2)}{(\int h_n dm+\norm{h-h_n}_1)^2(\int h_n)^2}\right\rVert_1.\\
\end{align*}

This allows us to bound the $L^1$ error of the normalised linear response by
\begin{align}\label{eq:finalbound}
&\frac{\norm{h^*-h_\eta^*}_1}{\norm{h_n}_1}+\frac{(\norm{h_\eta^*}_1+\norm{h^*-h^*_\eta}_1)\norm{h-h_n}_1}{(\norm{h_n}_1-\norm{h-h_n}_1)\norm{h_n}_1}\\
+&\frac{\norm{h-h_n}_1(\norm{h^*_\eta}_1-\norm{h^*-h^*_\eta}_1)}{(\norm{h_n}_1+\norm{h-h_n}_1)^2}+\frac{\norm{h_n}_1\norm{h^*-h^*_\eta}_1}{(\norm{h_n}_1+\norm{h-h_n}_1)^2}\nonumber\\
+&\frac{2\norm{h_\eta^*}_1\norm{h-h_n}_1}{(\norm{h_n}_1+\norm{h-h_n}_1)^2}+\frac{\norm{h_\eta^*}_1\norm{h-h_n}_1^2}{(\norm{h_n}_1+\norm{h-h_n}_1)^2\norm{h_n}_1}\nonumber
\end{align}

\section{Proof of Theorem \ref{thm:main}}\label{sec:mainthm}
In this section we give a proof of the main result in the paper, i.e., that we can approximate as well
as we want the linear response.
\begin{proof}[Proof of Theorem \ref{thm:main}]
Using \eqref{eq:approxh^*} we have
\begin{equation}
\begin{split}
\|h^*-h^*_\eta\|_1&\le\|F_0\hat h^*-F_0^{\text{app}}\hat h_\eta^*\|_1+\|Q\hat h-Q^{\text{app}}\hat h_n\|_1\\
&\le\|F_0\hat h^*-F_0\hat h^*_\eta\|_1+\|F_0\hat h^*_\eta-F_0^{\text{app}}\hat h_\eta^*\|_1\\
&\hskip 0.5cm+\|Q\hat h-Q\hat h_n\|_1+\|Q\hat h_n-Q^{\text{app}}\hat h_n\|_1\\
&:=(I)+(II)+(III)+(IV).
\end{split}
\end{equation}
By \eqref{theF}, we get
\begin{equation}\label{eq:estI}
\begin{split}
(I)&\le \|\hat h^*-\hat h^*_\eta\|_1+ \sum_{\omega\in\Omega}\int_{\Delta^c}\left|\left(\hat h^*\circ g_{\omega}-\hat h^*_\eta\circ g_{\omega}\right)g_{\omega}'\right|dx\\
&\le\|\hat h^*-\hat h^*_\eta\|_1+\|\hat h^*-\hat h^*_\eta\|_{C^0}\cdot \sum_{\omega}\|g_\omega'\|_{\mathcal B}\int_{\Delta^c}x^{-\gamma} dx\\
&\le\|\hat h^*-\hat h^*_\eta\|_1+\|\hat h^*-\hat h^*_\eta\|_{C^0}\cdot \frac{1}{2^{1-\gamma}(1-\gamma)} \cdot \sum_{\omega}\|g_\omega'\|_{\mathcal B}.
\end{split}
\end{equation}
Using \eqref{theF} again, we have
\begin{equation}\label{eq:estII}
\begin{split}
(II)&\le  \sum_{|[\omega]|>N^*}\int_{\Delta^c}\left|\hat h^*_\eta\circ g_{\omega}g_{\omega}'\right|dx\le\|\hat h^*_\eta\|_{C^0}\cdot \sum_{|[\omega]|>N^*}\|g_\omega'\|_{\mathcal B}\int_{\Delta^c}x^{-\gamma} dx\\
&\le \frac{1}{2^{1-\gamma}(1-\gamma)}\cdot\|\hat h^*_\eta\|_{C^0}\cdot \sum_{|[\omega]|>N^*}\|g_\omega'\|_{\mathcal B}.
\end{split}
\end{equation}
Note that by \eqref{a4'}, one can choose $N^*$ large enough so that $(II)$ is sufficiently small. Using \eqref{eq:Q}, we have
\begin{equation*}
\begin{split}
(III)&\le \sum_{\omega\in\Omega}\int_{\Delta^c}\left |\left(\hat h'\circ g_\omega-\hat h_n'\circ g_\omega\right) \cdot a_\omega g_\omega' \right|dx+ \sum_{\omega\in\Omega}\int_{\Delta^c}\left|\left(\hat h\circ g_\omega-\hat h_n\circ g_\omega\right)\cdot b_\omega\right|dx.
\end{split}
\end{equation*}
Now using \eqref{a4.0}, \eqref{a5}, and the change of variables $y_\omega=g_{\omega}(x)$ we get\begin{equation}\label{eq:estIII}
\begin{split}
(III)&\le\sup_{\omega}|a_\omega| \sum_{\omega\in\Omega}\int_{[\omega]_0}\left |\hat h'( y_\omega)-\hat h_n'(y_\omega) \right|dy_\omega+ \|\hat h-\hat h_n\|_{C^0}\cdot \sum_{\omega\in\Omega}\| b_\omega\|_{\mathcal B}\int_{\Delta^c}x^{-\gamma} dx\\
&=\sup_{\omega}|a_\omega|\cdot\|\hat h'-\hat h'_n\|_1+\frac{1}{2^{1-\gamma}(1-\gamma)}\cdot \sum_{\omega\in\Omega}\| b_\omega\|_{\mathcal B}\cdot \|\hat h-\hat h_n\|_{C^0}\\
&\le \max\left\{\sup_{\omega}|a_\omega|, \frac{1}{2^{1-\gamma}(1-\gamma)}\cdot \sum_{\omega\in\Omega}\| b_\omega\|_{\mathcal B}\right\}\cdot \|\hat h-\hat h_n\|_{C^1}.
\end{split}
\end{equation}
Finally, using \eqref{eq:Q} again, we have
\begin{equation*}
\begin{split}
(IV)&\le \sum_{|[\omega]|>N^*}\int_{\Delta^c}\left|\hat h_n'\circ g_\omega \cdot a_\omega g_\omega' \right|dx+ \sum_{|[\omega]|>N^*}\int_{\Delta^c}\left|\hat h_n\circ g_\omega\cdot b_\omega\right|dx.
\end{split}
\end{equation*}
Using \eqref{a4.0} and \eqref{a4'} in the first integral, and using \eqref{a5} in the second integral, we choose $N^*$ large enough and get
\begin{equation}\label{eq:estIV}
(IV)\le \frac{1}{2^{1-\gamma}(1-\gamma)}\left[\sup_{\omega}|a_\omega|\cdot \|\hat h'_n\|_{C^0}\cdot \sum_{|[\omega]|>N^*}\|g_\omega'\|_{\mathcal B}+\|\hat h_n\|_{C^0}\cdot \sum_{|[\omega]|>\in N^*}\| b_\omega\|_{\mathcal B}\right].
\end{equation}
Choosing $l^*$ in \ref{eq:inducedresponseapprox} to make $\sum_{n=l^*+1}\norm{\hat L^nW}_{C^1}$ small enough, followed by $k$ and $\eta$ to make \eqref{eq:estI} and \eqref{eq:estIII} small enough, then choosing $N^*$ in \eqref{eq:estII} and \eqref{eq:estIV} so $(I)+(II)+(III)+(IV)\le \tau$ completing the proof.
\end{proof}

\section{Application to an example}\label{sec:example}
In this section we will apply our algorithm to a classical example of maps with
an indifferent fixed point, strictly related to Pomeau-Manneville maps, the Liverani-Saussol-Vaienti
map.
The behaviour of this map is determined by the exponent $\alpha$; if $\alpha\in (0, 1)$
it is a non-uniformly expanding map with an absolutely continuous invariant probability measure;
if $\alpha\geq 1$ there is an absolutely continuous invariant infinite measure.

\subsection{Definition of the map and the induced map}\label{s:IMDef}

The equation of the map is 
\begin{equation}\label{eq:PMmap}
   T(x)=\begin{cases}
   x(1+2^{\alpha}x^{\alpha})\mbox{ if $x\in[0,\frac{1}{2}]$}\\
   2x-1 \mbox{ if $x\in(\frac{1}{2},1]$}
   \end{cases}.
\end{equation}

\begin{numassumption}
We fix $\alpha=\datax{alph}$ in our example. This is the value 
corresponding to $\epsilon=0$ in the previous section.
\end{numassumption}

\begin{figure}
   \begin{subfigure}[b]{0.45\textwidth}
     \includegraphics[width=55mm,height=50mm]{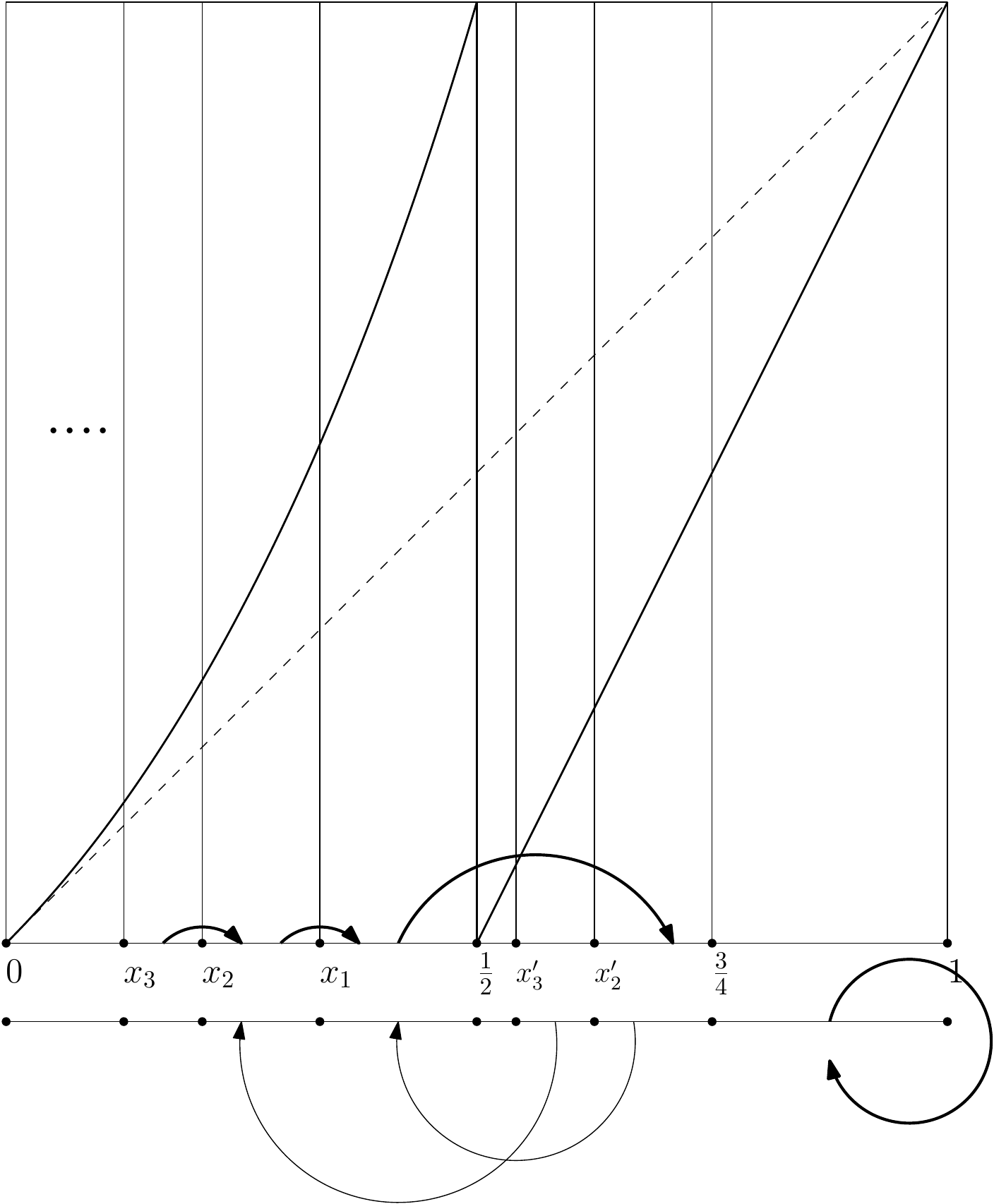}
    \caption{The map $T$}
   \end{subfigure}
   \begin{subfigure}[b]{0.45\textwidth}
    \includegraphics[width=55mm,height=50mm]{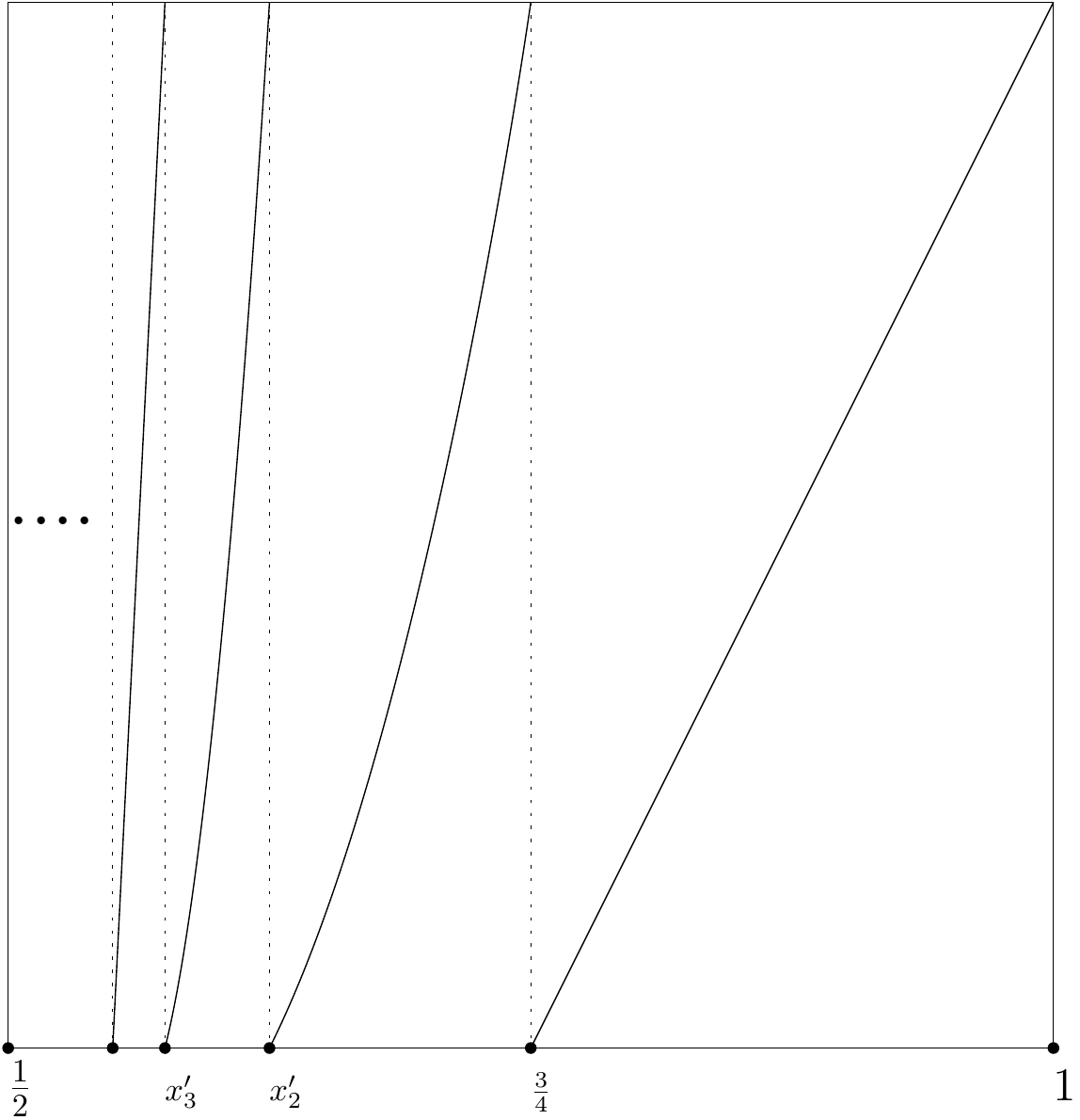} 
    \caption{The induced map $\hat T$}
   \end{subfigure}  
   \caption{The inducing scheme for $T$ in \eqref{eq:PMmap}.}
   \label{fig:PMinducing}
 \end{figure}
 
We construct the inducing scheme as in Subsection \ref{ss:T}.
Let  $x'_0=1$, $x'_1=\frac{3}{4}$,  and
$$
x'_n=g_{\omega}(\frac12) \,\, \text{for} \,\, n\ge 2.
$$  
Letting $\omega = 10^n$ and $g_{\omega}=g_{1}\circ g_{0}^{n}$,
then cylinder set $[\omega]$ is given by
$g_{\omega}([0.5, 1]) = (x'_n, x'_{n-1}]$.

Then $\hat T:\Delta\to\Delta$ is a piecewise smooth and onto map with countable number of branches and it satisfies all the assumptions of subsection \ref{ss:T}. See Figure \ref{fig:PMinducing} for a pictorial representation of the above inducing scheme.

\subsubsection{Numerical remark: the Shooting Method}\label{ShootingMethod}
To approximate rigorously the operators in this paper we need a rigorous way to approximate
long orbits given a coding, i.e.,
we need to be able to compute
\[
x = g_{\omega} (y) = g_1\circ g_0^{n-1}(y)
\]
i.e. we need to be able to compute $x\in [0.5, 1]$ such that
\[
T^{n+1}(x) = y, \quad T^i(x)\in [0, 0.5]
\]
for $i\in\{1,\dots,n\}$.

To solve this problem efficiently and obtain tight bounds on $x$ is tricky taking preimages sequentially leads to propagation of errors and the computed interval ends up being not usable.

The main idea is to substitute the equation above with the following system of equations
(this tecnique is called the Shooting Method, and we were introduced to it by W. Tucker)
\[
\left\{
\begin{array}{cc}
T(x_1) - x_2 = 0 & x_1\in [0.5, 1] \\
T(x_2) - x_3 = 0 & x_2\in [0, 0.5] \\
T(x_3) -x_4 = 0 & x_3\in [0, 0.5] \\
\vdots & \\
T(x_n) -y = 0  & x_n\in [0, 0.5].
\end{array}
\right.
\]

We will use the rigorous Newton method \cite{Tucker}, to simultaneously enclose $x_1, \ldots, x_n$. This way
we are solving a unique system of equations instead of propagating backwards the error through solving equations with a ``fat'' variable.
Given a function $\phi:\mathbb{R}^n\to \mathbb{R}^n$ and a vector of intervals $\hat{x} = (\hat{x}_1, \ldots, \hat{x}_n)$
the rigorous Newton step is given by
\[
N(\hat{x}) = \hat{x} \cap \textrm{mid}(\hat{x})-D\phi(\hat{x})^{-1}\phi(\textrm{mid}(\hat{x})),
\]
where the intersection between interval vectors  is meant componentwise and $\textrm{mid}$ is a function
that sends a vector of intervals to the vector of their midpoints \cite{Tucker}.

In our specific case, the shooting method is numerically well behaved: denoting by
$\phi(x_1, \ldots, x_n) = (T(x_1)-x_2, \ldots, T(x_n)-y)^T$, the Jacobian $D\phi$ is given by
a bidiagonal matrix, whose $i-th$ diagonal entry is $T'(x_i)$ and the superdiagonal entries are constant
and equal to $-1$.
In particular, this guarantees us that the Jacobian is invertible, since its eigenvalues correspond
to the diagonal elements and these are bounded away from $0$.
Moreover a bidiagonal system is solved in time $O(n)$ by backsubstitution, with small numerical error,
and these assumptions guarantee that the interval Newton method converges.

This allows us to compute tight enclosure of $g_{\omega} (y)$, $g'_{\omega} (y)$, which allows us to compute
discretizations of the transfer operator.

\subsection{Computing the error when taking a finite number of branches}\label{ss:ApproxPM}
Since we cannot calculate values for maps with infinitely many branches on the computer we use an approximating map as described in subsection \ref{ss:ApproxMethod},
\begin{figure}
   \includegraphics[width=55mm,height=50mm]{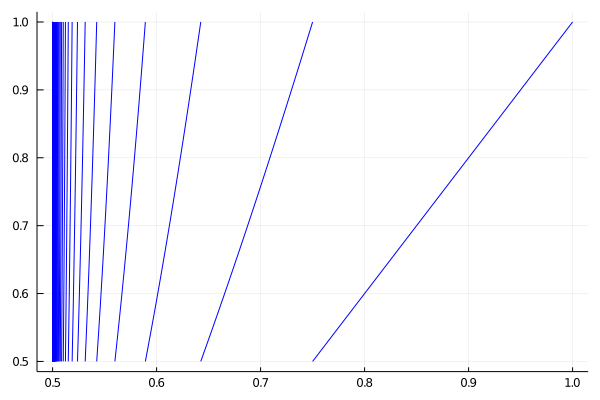} 
   \caption{The approximation map with $k=\datax{branches}$}
   \label{fig:ApproxMap}
\end{figure}
 this is depicted in figure \ref{fig:ApproxMap} for $\alpha=\datax{alph}$. To calculate bounds on the $C^1$ distance between the systems these maps define we use lemma \ref{lem:tripleC2} and find $D$ and $D_0$ for LSV maps. 
 Estimating these bounds efficiently is delicate since it involves estimating 
 the sum (and the tail) of converging series whose general term is going to zero slowly.
 The estimates in literature \cite{BS, K} give rise to values that are impractical for our computations; as an example, the value of the constant $C_8$ in \cite{K} computed according
 to their proof is of the order of $10^{269}$, which makes its use in our computations unfeasible, therefore some work is needed to give sharper bounds for the constants.
 Since these estimates are quite technical and need the introduction of specific notations, we separate them in the appendix not to hinder the flow of the sections.

In section \ref{a:DD0numbers} we bound $D_0\leq\datax{D0}$ and $D\leq\datax{D}$, so 
\begin{equation*}
\norm{(\hat L-\hat L_{\delta_k})}_{C^1}\leq (D+D_0D+2)\delta_k\leq \text{\datax{Lemma31Const}}\delta_k
\end{equation*}
can be made as small as needed by increasing $k$. 

Choosing $k=\datax{branches}$ gives 
$\norm{\hat L-\hat L_{\delta_k}}_{C^1}\leq$ \datax{L31}.

\subsubsection{Bounding $\norm{\hat h_{\delta_k}-\hat h}_{C^1}$}
In section \ref{a:Cstar} we prove the following Lasota-Yorke inequality
\[\norm{\hat L^n_{\delta_k}f}_{C^1}\leq \text{$\datax{AppendixM}$}\cdot (0.5)^n\norm{f}_{C^1}+\text{$\datax{AppendixM2}$}\norm{f}_\infty\].

The Lasota-Yorke inequality implies that 
\[
\norm{\hat L^n_{\delta_k}f}_{C^1}\leq \textrm{$\datax{AppendixCstar}$},
\]
which together with Lemma \ref{rem:approx1}, and the fact proved in section \ref{a:C1Contraction} by using the methods from \cite{GNS}  that 
\[
   \norm{\hat L_{\delta_k}^{\text{$\datax{C1N}$}}|_{U^0}}_{C^1}\leq \text{$\datax{LkC1contraction}$}
\]
allowing us to bound
\begin{equation}\label{eq:deltaherr}
\norm{\hat h-\hat h_{\delta_k}}_{C^1}\leq \frac{\text{\datax{C1N}}\cdot\text{\datax{AppendixCstar}}}{1-\text{\datax{LkC1contraction}}}\norm{(\hat L_{\delta_k}-\hat L)\hat h}_{C^1}\leq \text{\datax{LkLC1}}\norm{\hat h}_{C^1}
\end{equation}
Observing that 
$\norm{\hat h}_{C^1}\leq\norm{\hat h_{\delta_k}}_{C^1}+\norm{\hat h-\hat h_{\delta_k}}_{C^1}$ 
and a bound on $\norm{\hat h_{\delta_k}}_{C^1}$ in section \ref{a:Cstar} 
gives us a final error of \datax{h_error_1}.

\subsubsection{Computing the discretization error}\label{ss:DiscOp}
The truncated operator $\hat L_{\delta_k}$ satisfies the following 
Lasota-Yorke like inequalities\footnote{it is straightforward to see that these inequalities
imply Lasota-Yorke inequalities on $W^{k,1}$ with weak norm $W^{k-1,1}$.}
 \begin{align*}
||(\hat L_{\delta_k}^n)f'||_1\leq& \lambda^n ||f'||_1+1.785 ||f||_1\\
||(\hat L_{\delta_k}^n)f''||_1\leq &\lambda^{n}||f''||_1+0.3076||f'||_1+6.57||f||_1\\
||(\hat L_{\delta_k}f)'''||_1 \leq &\lambda^{3n}||f'''||_1+0.145||f''||_1+1.98||f'||_1+36.96||f||_1 \\
||(\hat L_{\delta_k}^nf)^{(4)}||_1 \leq& \lambda^{4n}||f^{(4)}||_1+0.057||f^{(3)}||_1+1.49||f''||_1+16.97||f'||_1+559.4||f||_1\\
||(\hat L_{\delta_k}^nf)^{(5)}||_1\leq& \lambda^{5n}||f^{(5)}||_1+0.0199||f^{(4)}||_1+0.85||f^{(3)}||_1\\
+&17.57||f''||_1 +794.59||f'||_1+10086||f||_1 \\
||(\hat L_{\delta_k}^nf)^{(6)}||_1\leq &\lambda^{6n}||f^{(6)}||_1+0.0066||f^{(5)}||_1+0.41||f^{(4)}||_1\\
+&13.33||f^{(3)}||_1+895||f''||_1 +24840.2||f'||_1+684431||f||_1.
\end{align*}

Since we know $\norm{h_{\delta_k}}_1=1$ for $f$ a probability density, we can use these to get a bound on $\norm{h_{\delta_k}^{(6)}}_1$ which is calculated to be \datax{V}. 
Denoting by $\hat L_n = \pi_n \hat L_{\delta_k} \pi_n$ the discretized operator on the base 
of Chebyshev polynomials of the first kind of degree up to $n$, the
same Lasota-Yorke inequalities allow us to compute 
\[
\norm{\hat L_{\delta_k}-\hat L_n}_{C^1}\leq \text{\datax{LYerror}}.
\]

This, together with the computed bounds on the $C^1$ mixing rate in table \ref{tab:C1Contraction}
\begin{table}[h!]
  \begin{center}
    \caption{Calculated contraction rates of our discretised operators.}
    \label{tab:C1Contraction}
    \begin{tabular}{l|c} 
      \textbf{k} & $\norm{\hat L_n^k|_{U^0}}_{C^1}$ \\
      \hline
      1 & \datax{cc1}\\
      2 & \datax{cc2}\\
      3 & \datax{cc3}\\
      4 & \datax{cc4}\\
      5 & \datax{cc5}\\
      6 & \datax{cc6}\\
      7 & \datax{cc7}\\
      8 & \datax{cc8}\\
      9 & \datax{cc9}\\
      10 & \datax{cc10}\\
      11 & \datax{cc11}\\

    \end{tabular}
  \end{center}
\end{table}
%
gives us an error of
\begin{equation*}
\norm{\hat h-\hat h_n}_{C^1}\leq\text{\datax{heta_error}};
\end{equation*}
in figure \ref{fig:ApproxDensity} a plot of the approximated density is presented.
\begin{figure}
   \includegraphics[width=55mm]{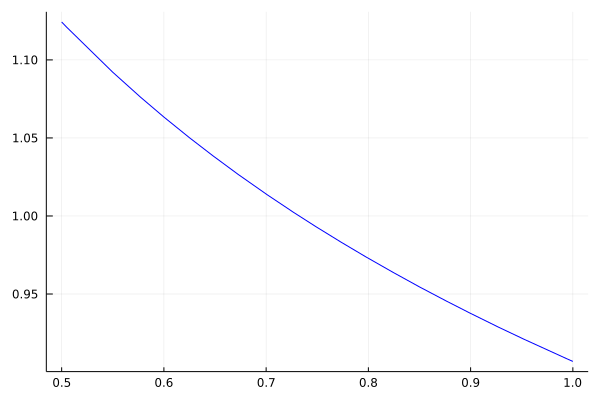}
   \caption{The invariant density of the induced map as calculated according to section \ref{ss:ApproxMethod}.}
   \label{fig:ApproxDensity}
\end{figure}
\subsubsection{Numerical remark: automated Lasota-Yorke inequalities}\label{ss:autoly}
We detail a way to automatically calculate Lasota-Yorke type inequalities for transfer operators in $W^{k,1}$.
Following \cite{BKL} let
\[
 L_k f = \sum_{y\in T^{-1}(x)}\frac{f(y)}{|T'(y)|^k}.
\]
From this follows
\begin{equation}\label{eq:der}
(L_k f)' =  L _{k+1} f' + k  L _{k}(f D), 
\end{equation}
where $D = (1/T'(x))'$ is the distorsion.
We use the formula above to compute symbolical expressions for the derivatives
$(L_1 f)^{(l)}$. 

We use Interval Arithmetic and higher order Automatic Differentiation \cite{Tucker}(as implemented in TaylorSeries.jl) to compute bounds for 
\[
\norm{(1/T')^{(l)}}_{\infty}.   
\]
This allows us to bound the coefficients of the Lasota-Yorke inequalities.

\subsection{Approximating the linear response for the induced map}\label{ss:IPMResponse}
\begin{figure}
   \includegraphics[width=55mm]{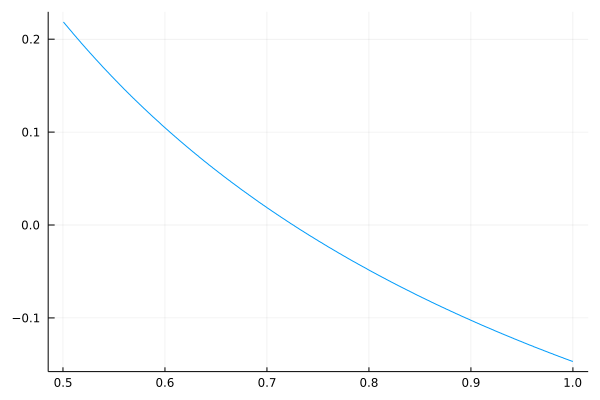} 
   \caption{The linear response of the induced map as calculated according to Section \ref{sec:InducedResponse}.}
   \label{fig:ApproxResponse}
\end{figure}
We approximate the linear response of our induced map using \eqref{eq:inducedresponseapprox} which uses the 
discretized operator $\hat P_h$, associated to the hat discretization of size $\eta=$\datax{eta}. 
We get our error from lemma \ref{thm:induced_resp}, which gives us four terms that need to be bound, each of which is done in the appendix, \ref{a:IR1}, \ref{a:IR2}, \ref{a:estimateshstar3} and \ref{a:estimateshstar4} for $l^*=$\datax{lstar} and $k=$\datax{branches}:\marginnotet{Some changes here}
\begin{enumerate}
\item $\sum_{n=0}^{l^*}\sum_{i=0}^n\norm{\hat L^i|_{U^0}}_{C^0}\norm{(\hat L_w-\hat L_{\delta_k})\hat L_w^{n-i}W_h}_{C^0} \leq $\datax{hstar_error_1}
\item $\sum_{n=1}^{l^*}\sum_{i=0}^{n-1}\norm{\hat L^i|_{U^0}}_{C^0}\norm{(\hat L_{\delta_k}-\hat L)\hat L_w^{n-i}W_h}_{C^0}\leq $\datax{hstar_error_2}
\item $\sum_{n=0}^{l^*}\norm{\hat L^n|_{U^0}}_{C^0}\norm{W_h-W}_{C^0}\leq$\datax{hstar_error_3}
\item $\sum_{n=l^*+1}^{\infty}\norm{\hat L^n W}_{C^0}\leq$\datax{hstar_error_4}
\end{enumerate}
this gives us a total error $\norm{\hat h^*- \hat h_\eta^*}_{C^0}\leq$ \datax{hetastar_error}

\subsubsection{The contraction rates of $\hat L_{\delta_k}$ in the $C^0$ norm.}\label{a:C1Contraction}
In order to bound $\norm{\hat L^n_{\delta_k}|_{U^0}}_{C^0}$ we use lemma 7.13 from \cite{BGNN} to bound
\begin{equation*}
\norm{(\hat L_{\delta_k}^m-\hat L_n^m)f}_{C^0}\leq A_m\norm{f}_{C^1}+B_m\norm{f}_{C^0} 
\end{equation*}
from which we can use $C_{c,m}$ to bound $\norm{\hat L_{\delta_k}^m|_{U^0}}_{C^0}\leq A\norm{f}_{C^1}+(B+C_{c,m})\norm{f}_{C^0} $
and the Lasota-Yorke inequality \eqref{eq:LYC2} from section \ref{a:Cstar}, and use the small matrix method from \cite{GN}. We have
\[
\begin{pmatrix}
\norm{\hat L^m_{\delta_k} f}_{C^1}\\
\norm{\hat L^m_{\delta_k} f}_{C^0}
\end{pmatrix}\leq
\begin{pmatrix}
M\lambda^{m} & D\\
A_n & B_n+C_{c,m}
\end{pmatrix}
\begin{pmatrix}
\norm{f}_{C^1}\\
\norm{f}_{C^0}
\end{pmatrix}.
\]
Choosing the value of $n$ that minimises equation \eqref{eq:deltaherr}, we take $C_{c,4}\leq$\datax{cc4} from section \ref{ss:DiscOp}, together with the calculation $\norm{(\hat L_{\delta_k}^{4}-\hat L_n^{4})f}_{C^0}\leq $\datax{LB131}$\norm{f}_{C^1}+$\datax{LB132}$\norm{f}_{C^0}$, which gives the largest eigenvalue of the small matrix $\rho=$\datax{LkC1contraction}.

\subsubsection{The contraction rates of $\hat L$ in the $C^0$ norm.}\label{a:InducedC1Contraction}
We can use $\norm{\hat L_{\delta_k}^{n}|_{U^0}}_{C^0}\leq \rho_{n,C^0}$ and parts of lemma \ref{lem:tripleC2} to get 
\begin{equation*}
\norm{\hat L^n|_{U^0}}_{C^0}\leq \norm{\hat L_{\delta_k}^{n}|_{U^0}}_{C^0}+\norm{\hat L^n_{\delta_k}-\hat L^n}_{C^0}\leq nC^*(2+D)\delta_k+\rho_{n,C^0}.
\end{equation*}
These give us $\norm{\hat L^{n}|_{U^0}}_{C^0}\leq$\datax{lstarcontract} for $n=$\datax{lstarn}.

\subsubsection{Bounding Item (1)}\label{a:IR1}
We can bound  $\sum_{n=0}^{l^*}\sum_{i=0}^n\norm{\hat L^i|_{U^0}}_{C^0}\norm{(\hat P_h-\hat L_{\delta_k})\hat P_h^{n-i}W_h}_{C^0}$ by
%
\begin{align*}
&\sum_{n=0}^{l^*}\sum_{i=0}^n\norm{\hat L^i (\hat P_h-\hat L_{\delta_k})\hat P_h^{n-i}W_h}_{C^0}\\
\leq&\eta K\sum_{n=0}^{l^*}\sum_{i=0}^n \norm{\hat L^i|_{U^0}}_{C^0}(\textrm{Lip}(\hat P_h^{n-i}W_h) + \norm{\hat P_h^{n-i}W_h}_{C^0})
\end{align*}
where $K=2(D+1+\lambda(D+1))$.

We can calculate $\textrm{Lip}{(\hat P_h^{n-i}W_h})$ and $\norm{\hat P_h^{n-i}W_h}_{C^0}$
explictly by using validated numerical methods, since 
$W_h$ explicitly represented on the computer, so we can compute an enclosure of $\hat P_h W_h$ by rigorous matrix multiplication; for a function $f$ in the hat basis with coefficients $v_i$ we have explicit functions that allow us to compute
the $\textrm{Lip}$ and $C^{0}$ norm, i.e.: 
\[
\quad \textrm{Lip}(f) = \max \frac{|v_{i+1}-v_i|}{\eta}, \quad ||f||_{C^0} = \max_i |v_i|.
\] 
%
%
 We calculate this to give
\begin{equation*}
\sum_{n=0}^{l^*}\sum_{i=0}^n\norm{\hat L^i (\hat P_h-\hat L_{\delta_k})\hat P_h^{n-i}W_h}_{C^0}\leq\text{\datax{hstar_error_1}}
\end{equation*}
\subsubsection{Bounding Item (2)}\label{a:IR2}
To bound $\sum_{n=1}^{l^*}\sum_{i=0}^{n-1}\norm{\hat L^i(\hat L_{\delta_k}-\hat L)\hat P_h^{n-i}W_h}_{C^0}$, observe
\begin{align*}
&\sum_{n=0}^{l^*}\sum_{i=0}^n\norm{\hat L^i(\hat L_{\delta_k}-\hat L)\hat P_h^{n-i}W_h}_{C^0}\\
\leq&\sum_{n=0}^{l^*}\sum_{i=0}^n\norm{\hat L^i}_{C^0}\norm{(\hat L_{\delta_k}-\hat L)\hat P_h^{n-i}W_h}_{C^0}
\end{align*}
and
\begin{align*}
&\norm{(\hat L_{\delta_k}-\hat L)f}_{C^0}\\
\leq&\norm{f(\delta_k(2x-1)+\frac 1 2)2\delta_k-\sum_{|\omega|>N^*}f\circ g_\omega|g_\omega'|}_{C^0}\\
\leq&2\delta_k\norm{f(\delta_k(2x-1)+\frac 1 2)}_{C^0}+\sum_{|\omega|>N^*}\norm{f\circ g_\omega |g_\omega'|}_{C^0}\\
\leq&(2\delta_k+\sum_{|\omega|>N^*}|g_\omega'|)\norm{f}_{C^0}
\end{align*}
and therefore we have an upper bound of
\begin{equation*}
(2\delta_k+\sum_{|\omega|>N^*}|g_\omega'|)\sum_{n=0}^{l^*}\sum_{i=0}^n\norm{\hat L^i}_{C^0}\norm{\hat P_h^{n-i}W_h}_{C^0}
\end{equation*}
As in the estimate for item (1) we can compute $\norm{\hat P_h^{n-i}W_h}_{C^0}$ explicitly, which gives us 
\[
\sum_{n=1}^{l^*}\sum_{i=0}^{n-1}\norm{\hat L^i(\hat L_{\delta_k}-\hat L)\hat P_h^{n-i}W_h}_{C^0}\leq
\text{\datax{hstar_error_2}}.
\]
%
%
%
%
%
\subsubsection{Bounding Item (3)}\label{a:estimateshstar3}
We bound $\norm{W_h-W}_{C^0}$ by the following
\begin{equation*}
\norm{W_h-W}_{C^0}\leq\norm{W_h-W_{\infty}}_{C^0}+2\textrm{Lip}(W)\eta,
\end{equation*}
as explained in remark \ref{rem:WW_h}.

Remark that $\textrm{Lip}(W)=\norm{W'}_{\infty}$ that we computed in section 
\ref{a:WprimeBound} to be bounded by \datax{WprimeBound}.

Recall that 
\[
W_h = \sum_{i=0}^h a_i \phi_i(x)   
\]
and
\[
W_{\infty} = \sum_{i=0}^h \tilde{a}_i \phi_i(x)   
\]
where 
\[
a_i =\sum_{\omega\leq k} \frac{1}{\eta}\int_{g_\omega(I_i)}A_0(z)\cdot \hat h_n'(\zeta)+B_0(z)\hat h_n(\zeta)dz+\frac{1}{\eta}\int_{\delta_k(2I_i-1)+1/2}[A_0\hat h'_\eta+B_0\hat h_n](z)dz   
\]
and 
\[
\tilde{a}_i = \int_{I_i} \hat L (A_0 h'+B_0 h) = \sum_{\omega} \frac{1}{\eta}\int_{g_\omega(I_i)}A_0(x)\cdot \hat h'(x)+B_0(x)\hat h(x)dx
\]

Then we have that 
\[
||W_h-W_{\infty}||_{C^0}\leq \sup|a_i-\tilde{a}_i|.   
\]

We introduce auxiliary quantities 
\[
a_{i,*}= \sum_{\omega} \frac{1}{\eta}\int_{g_\omega(I_i)}A_0(x)\cdot \hat h'(\zeta)+B_0(x)\hat h(\zeta)dx
\]
where the $\zeta$ in $g_{\omega}(I_i)$ is chosen as in the computation of $a_i$ for all
$i$ and for all $\omega<k$, and arbitrarily for all $\omega\geq k$.

Then 
\begin{align*}
&|\tilde{a}_{i}-a_{i,*}|= \left|\sum_{\omega} \frac{1}{\eta}\int_{g_\omega(I_i)}A_0(x)\cdot \hat h'(x)+B_0(x)\hat h(x)dx-\int_{g_\omega(I_i)}A_0(x)\cdot \hat h_n'(\zeta)+B_0(x)\hat h_n(\zeta)dx\right|\\
=&\left|\sum_{\omega} \frac{1}{\eta}\int_{g_\omega(I_i)}A_0(x)\cdot (\hat h'(x)-\hat h'_n(\zeta))+B_0(x)(\hat h(x)-\hat h_n(\zeta))dx\right|\\
\leq&\sum_{\omega} (|g_{\omega}(I_i)|\norm{\hat h''}_{\infty}+\norm{\hat h'-\hat h_n'}_{\infty})\frac{1}{\eta}\left|\int_{g_\omega(I_i)}A_0(x)dx\right|\\
&+(|g_{\omega}(I_i)|\norm{\hat h'}_{\infty}+\norm{\hat h-\hat h_n}_{\infty})\frac{1}{\eta}\left|\int_{g_\omega(I_i)}B_0(x)dx\right|.
\end{align*}


Recalling that $A_0=\partial_{\epsilon}g_\omega(\hat T_{\epsilon})$ $B_0=\partial_{\epsilon} g_\omega'(\hat T_\epsilon)\cdot \hat T'_\epsilon$ 
\footnote{Note that while this appears to depend on $\omega$, we choose the value of $\omega$ coresponding to the appropriate branch of $\hat T_\epsilon$.}  
we get
\begin{align*}
&\sup_{i}\sum_{\omega} (|g_{\omega}(I_i)|\norm{\hat h''}_{\infty}+\norm{\hat h'-\hat h_n'}_{\infty})\frac{1}{\eta}\left|\int_{g_\omega(I_i)}A_0(x)dx\right|\\
&+(|g_{\omega}(I_i)|\norm{\hat h'}_{\infty}+\norm{\hat h-\hat h_n}_{\infty})\frac{1}{\eta}\left|\int_{g_\omega(I_i)}B_0(x)dx\right|\\
\leq&\sup_{i}\sum_{\omega} (|g_{\omega}(I_i)|\norm{\hat h''}_{\infty}+\norm{\hat h'-\hat h_n'}_{\infty})\frac{1}{\eta}\left|\int_{I_i}\partial_{\epsilon}g_\omega(x)\hat T_{\epsilon}'(x)\right|\\
&+(|g_{\omega}(I_i)|\norm{\hat h'}_{\infty}+\norm{\hat h-\hat h_n}_{\infty})\frac{1}{\eta}\left|\int_{I_i}\partial_{\epsilon} g_\omega'(x)dx\right|\\
\leq&(|g_\omega(I_i)|\norm{\hat h}_{C^2}+\norm{\hat h-\hat h_n}_{C^1})\sum_{\omega}|\partial_{\epsilon}g_\omega\hat T_{\epsilon}'|+|\partial_{\epsilon} g_\omega'|
\end{align*}
which we calculate using methods from sections \ref{a:WprimeBound} and \ref{sec:A0B0} for $\eta=$\datax{eta} to get
\begin{equation*}
   |\tilde{a}_i-a_{i, *}|\leq\text{\datax{W1}}.
\end{equation*}

Now, denote by $b$ the linear branch of the truncated induced LSV map; then 
\[
|a_i-a_{i,*}| = \left| \frac{1}{\eta}\int_{b^{-1}(I_i)}A_0(y)h'_n(\zeta) + B_0(y)\hat h_n(\zeta)dy-\frac{1}{\eta}\sum_{|\omega|>k}\int_{g_{\omega}(I_i)}A_0(y) h'_n(\zeta)+B_0(y)\hat h_n(\zeta)dy\right|
\]
for all $i$.

We will only compute a bound for the $B_0$ terms, the $A_0$ terms are similar.

\begin{align*}
&\left|\frac{1}{\eta}\int_{b^{-1}(I_i)}B_0(y)\hat h_n(\zeta)dy-\frac{1}{\eta}\sum_{|\omega|>k}\int_{g_{\omega}(I_i)}B_0(y)\hat h_n(\zeta)dy\right|\\
\leq&\left|\frac{1}{\eta}\int_{b^{-1}(I_i)}B_0(y)\hat h_n(\zeta)dy\right|+\left|\frac{1}{\eta}\sum_{|\omega|>k}\int_{g_{\omega}(I_i)}B_0(y)\hat h_n(\zeta)dy\right|\\
\leq&\norm{\hat h_n}_{C^0}\left(\left|\frac{1}{\eta}\int_{b^{-1}(I_i)}B_0(y)dy\right|+\sum_{|\omega|>k}|\partial_\epsilon g_\omega'|\right)\\
\leq&\norm{\hat h_n}_{C^0}\left(\left|\frac{\delta_k}{2\eta}\int_{I_i}B_0(b^{-1}(y))dy\right|+\sum_{|\omega|>k}|\partial_\epsilon g_\omega'|\right)\\
\leq&\norm{\hat h_n}_{C^0}\left(\frac{\delta_k}{2\eta}\norm{B_0}_1+\sum_{|\omega|>k}|\partial_\epsilon g_\omega'|\right)\\
\end{align*}
which including the $A_0$ term in a similar fashion we get

\begin{equation*}
\norm{\hat h_n}_{C^1}\left(\frac{\delta_k}{2\eta}\max \{\norm{A_0}_1,\norm{B_0}_1\}+\sum_{|\omega|>k}(|\partial_\epsilon g_\omega'|+|\partial_{\epsilon}g_\omega\cdot g_\omega'|)\right).
\end{equation*}
Taking $k=$\datax{branches} this allows us to prove that $\norm{W_h-W}_{C^0}\leq$\datax{hstar_error_3}.

\subsubsection{Bounding $\norm{W'}_{C^0}$}\label{a:WprimeBound}
 For calculating $\norm{W'}_{C^0}$, first we observe that for the $B_0$ term
 \begin{align*}
(\hat L[(B_0\hat h)\circ g_\omega\cdot g_\omega'])'&=\sum([ \partial_{\epsilon}g'_\omega(T_\epsilon)T_\epsilon'\hat h]\circ g_\omega\cdot g_\omega')'\\
&=\sum (\partial_{\epsilon}g'_\omega(\hat h\circ g_\omega))'\\
&=\sum \partial_\epsilon g_\omega''(\hat h\circ g_\omega)+\partial_\epsilon g_\omega' \hat h'\circ g_\omega g_\omega'\\
&\leq \norm{\hat h}_{C^1}\max\{\sum \partial_\epsilon g_\omega'',\sum \partial_\epsilon g_\omega' g_\omega'\}
\end{align*}
from \cite{K} lemma 5.8 and 5.3 we have that 
\begin{equation*}
\sum \partial_\epsilon g_\omega''\leq C\sum _{r=1}^\infty \frac{4(1-\log(0.5))}{r^{(\alpha+1)/\alpha}}\leq C4(1-\log(0.5))\zeta''(9)
\end{equation*}
so it remains to find this $C$.

From the proof of lemma 5.8 we have
\begin{align*}
\partial_\epsilon z_r''=&[2\alpha+1+\alpha(\alpha+1)\log(2z_{r+1})]2^\alpha z_{r+1}^{\alpha-1}(z_{r+1}')^2\\
+&(\alpha-1)\alpha(\alpha+1)2^\alpha z_{r+1}^{\alpha-2}(z_{r+1}')^2\partial_\epsilon z_{r+1}+2\alpha(\alpha+1)2^\alpha z_{r+1}^{\alpha-1}z_{r+1}'\partial_\epsilon z_{r+1}'\\
+&(1+(\alpha+1)\log (2z_{r+1}))2^\alpha z_{r+1}^\alpha z_{r+1}''+\alpha(\alpha+1)2^\alpha z_{r+1}^{\alpha-1}z_{r+1}''\partial_\epsilon z_{r+1}\\
+&(1+(\alpha+1)2^\alpha z_{r+1}^\alpha)\partial_{\epsilon}z_{r+1}''
\end{align*}
which we divide by $z_{r}'=[1+(\alpha+1)2^\alpha z_{r+1}^\alpha]z_{r+1}'$ to get
\begin{align*}
\left | \frac{\partial_\epsilon z_r''}{z_{r}'}-\frac{\partial_\epsilon z_{r+1}''}{z_{r+1}'}\right |\leq&\frac{2\alpha+1+\alpha(\alpha+1)\log(2))}{\alpha+1}\left |z_{r+1}^{\alpha-1}z_{r+1}'\right |+\alpha\left |\log(z_{r+1})z_{r+1}^{\alpha-1}z_{r+1}'\right |\\
+&(\alpha-1)\alpha(\alpha+1)2^\alpha \left|z_{r+1}^{\alpha-2}z_{r+1}'\partial_\epsilon z_{r+1}\right|+2\alpha(\alpha+1)2^\alpha |z_{r+1}^{\alpha-1}\partial_\epsilon z_{r+1}'|\\
+&\frac{(2+(\alpha+1)\log(2)))2^\alpha}{1+(\alpha+1)2^\alpha }\left|\frac{ z_{r+1}''}{z_{r+1}'}\right|+\left|\frac {\log (z_{r+1})z_{r+1}''}{z_{r+1}'}\right|\\
+&\alpha \left|\frac{ z_{r+1}^{\alpha-1}z_{r+1}''\partial_\epsilon z_{r+1}}{z_{r+1}'}\right|
\end{align*}
so we must find bounds on $\left |z_{r+1}^{\alpha-1}z_{r+1}'\right |$, $\left |\log(z_{r+1})z_{r+1}^{\alpha-1}z_{r+1}'\right |$, $ \left|z_{r+1}^{\alpha-2}z_{r+1}'\partial_\epsilon z_{r+1}\right|$, $\left|\frac{ z_{r+1}''}{z_{r+1}'}\right|$, $\left|\frac {\log (z_{r+1})z_{r+1}''}{z_{r+1}'}\right|$ and $ \left|\frac{ z_{r+1}^{\alpha-1}z_{r+1}''\partial_\epsilon z_{r+1}}{z_{r+1}'}\right|$. 
We do this by \cite{K} lemmas 5.2, 5.3, 5.4, 5.6 and 5.7.

Lemma 5.2 gives
\begin{equation*}
\frac{C_1^{1/\alpha}z_0}{(r+1)^{1/\alpha}}\leq z_{r+1}\leq \frac{C_2^{1/\alpha}}{(r+1)^{1/\alpha}}.
\end{equation*}

Lemma 5.3 gives
\begin{equation*}
z_{r+1}'\leq C_8(1+(r+1)z_0^\alpha \alpha 2^\alpha)^{-(\alpha+1)/\alpha}\leq C_8 (\alpha 2^\alpha)^{-(\alpha+1)/\alpha}(r+1)^{-(\alpha+1)/\alpha}z_0^{-(\alpha+1)}.
\end{equation*}

Lemma 5.4 gives
\begin{equation*}
\left | \frac{z_{r+1}''}{z_{r+1}'} \right | \leq \alpha(\alpha+1)2^\alpha C_2^{(\alpha-1)/\alpha}C_8(\alpha 2^\alpha)^{-(\alpha+1)/\alpha} \frac{ z_0^{-2}}{\max \{r+1,1\}}
\end{equation*}

Lemma 5.6 gives
\begin{equation*}
|\partial_\epsilon z_r| \leq \frac{2^\alpha C_2^{1/\alpha}}{\alpha(\alpha+1)2^{\alpha -1}}\frac{logg(rz_0^\alpha)}{r^{1/\alpha}}(logg(r)-\log(z_0)).
\end{equation*}

Lemma 5.7 gives
\begin{equation*}
\left | \frac{\partial_\epsilon z_{r+1}'}{z_{r+1}'} \right |((\alpha+2)2^\alpha+2\frac{\alpha+1}{1-\alpha}C_2^{1/\alpha})[logg(rz_0^\alpha)]^2(logg(r)-\log(z_0))
\end{equation*}

Combining these gives
\begin{align*}
\left |z_{r+1}^{\alpha-1}z_{r+1}'\right |\leq&C_2^{(\alpha-1)/\alpha}C_8 (\alpha 2^\alpha)^{-(\alpha+1)/\alpha}(r+1)^{-2}z_0^{-(\alpha+1)}\\
\left |\log(z_{r+1})z_{r+1}^{\alpha-1}z_{r+1}'\right |\leq&C_3C_2^{(\alpha-1)/\alpha}C_8 (\alpha 2^\alpha)^{-(\alpha+1)/\alpha}(r+1)^{-2}z_0^{-(\alpha+1)}(logg(r+1)-\log(z_0))\\
\left |z_{r+1}^{\alpha-2}z_{r+1}'\partial_\epsilon z_{r+1} \right |\leq&\frac{C_2^{(\alpha-1)/\alpha}C_8(\alpha2^\alpha)^{(\alpha+1)/\alpha}}{\alpha(\alpha+1)2^{\alpha-1}}(r+1)^{-2}z_0^{-\alpha-1}logg(rz_0^\alpha)(logg(r)-\log (z_0))\\
\left | \frac{z_{r+1}''}{z_{r+1}'} \right | \leq& \alpha(\alpha+1)2^\alpha C_2^{(\alpha-1)/\alpha}C_8(\alpha 2^\alpha)^{-(\alpha+1)/\alpha} \frac{ z_0^{-2}}{\max \{r+1,1\}}\\
\left|\frac {\log (z_{r+1})z_{r+1}''}{z_{r+1}'}\right|\leq&\alpha(\alpha+1)2^\alpha C_2^{(\alpha-1)/\alpha}C_8(\alpha 2^\alpha)^{-(\alpha+1)/\alpha}C_3 \frac{ z_0^{-2}}{\max \{r+1,1\}}(\log(r+1)-\log(z_0))\\
 \left|\frac{ z_{r+1}^{\alpha-1}z_{r+1}''\partial_\epsilon z_{r+1}}{z_{r+1}'}\right|\leq&C_8^2(\alpha 2^\alpha)^{-2(\alpha-1)/\alpha}C_2^{2-1/\alpha}2^{\alpha-1}z_0^{\alpha-3}r^{-1/\alpha-3}.
\end{align*}
We use these to bound the tail and calculate rigorously a finite number of terms of
$\sum \partial_\epsilon g_\omega''$ allowing us to prove that 
$\sum \partial_\epsilon g_\omega''\leq$\datax{gppaSum} and $|I_i|\norm{W'}_\infty\leq$\datax{WprimeBound}.

\subsubsection{Bounding Item (4)}\label{a:estimateshstar4}
Using bounds on $\norm{\hat L|_{U^0}}_{C^0}\leq C_n$ from section \ref{a:InducedC1Contraction} we get
\begin{align*}
\sum_{n=l^*}^\infty \norm{\hat L^n W}_{C^0}\leq l^*\frac{C_{l^*}\norm{W}_{C^0}}{1-C_{l^*}}.
\end{align*}
To bound $\norm{W}_{C^0}$ we use
\begin{align*}
\norm{\hat L[A_0\hat h'+B_0\hat h]}_{C^0}\leq \norm{\hat h}_{C^1}\sum_{\omega}(|\partial_\epsilon g_\omega'|+|\partial_{\epsilon}g_\omega\cdot g_\omega'|).
\end{align*}
We have from subsection \ref{a:InducedC1Contraction} that $\norm{\hat L^{n}|_{U^0}}_{C^0}\leq$\datax{lstarcontract}, for $n=$\datax{lstarn} so we choose $l^*$ to be a multiple of \datax{lstarn} which gives for $l^*=$\datax{lstar}
\[
\sum_{n=l^*}^{\infty}\norm{\hat L^n W}_{C^0}\leq \text{\datax{hstar_error_4}}.
\]
%
%
\subsubsection{Calculating $\norm{A_0}_1$ and $\norm{B_0}_1$}\label{a:A0int}
For calculating $\norm{\hat h^*-\hat h_\eta^*}_1$ we need bounds on $\norm{A_0}_1$ and $\norm{B_0}_1$, as used in subsections \ref{a:estimateshstar3} and \ref{a:estimateshstar4}. We have a method to calculate the values of $A_0$ and $B_0$ from section \ref{sec:A0B0}, since $B_0=A_0'$ we can calculate the integral of $\int_{[a,b]} B_0(x) dx = A_0(b)-A_0(a)$. In order to calculate the integral of $A_0$ we approximate it by taking $k = \datax{A0Accuracy}$ evenly spaced values in each partition element $I_i$, we then take $\frac{|I_j|}{k}\sum_{j=1}^kA_0(x_j)$ as the value of the integral on $I_i$. This has an $L^1$ error of $\frac{|I_i|}{k}\textrm{Var} (A_0)$. Taking our approximation of the integral of $A_0$ and adding $\frac{|I_i|}{k}\textrm{Var}(A_0)$ gives an upper bound of $\norm{A_0}_1$.
\subsection{Pulling back to the original map}
To get the invariant density and linear response for the full map we must pull them back to the unit interval with $F$ and $Q$ from subsection \ref{ss:T}. The invariant density is fairly straight forward to calculate and find the error. We want a bound on $\norm{h-F_0^{app}\hat h_n}_1$ for which we can use a bound from (I) in the proof of theorem \ref{thm:main}.
\begin{equation*}
\norm{F_0\hat h-F_0^{app}\hat h_n}_1\leq 2\norm{\hat h_n-\hat h}_1+\frac{1}{2^{1-\gamma}(1-\gamma)}\norm{\hat h_n}_{C^0}\sum_{\omega>N^*}\norm{g'_\omega}_B
\end{equation*}
We use the bounds for $\gamma=$\datax{gamma} $\frac{1}{2^{1-\gamma}(1-\gamma)}\leq \datax{gamma_const}$, $\sum_{\omega>N^*}\norm{g'_\omega}_B\leq$\datax{gprime_tail}  as calculated in section \ref{a:gomegaprime}  and $\norm{\hat h_n}_{C^0} \leq $\datax{heta_max}, which gives the second term to be bounded by \datax{hpbII}. The first term we can bound by $2\norm{\hat h_n-\hat h}_{C^1}\leq$\datax{hpbI} as calculated in section \ref{ss:DiscOp}, giving us $\norm{h-F_0^{app}\hat h_n}_{1}\leq $\datax{hpb}.
%
%

As seen in theorem \ref{thm:main} pulling back the linear response requires the following bounds
\begin{align*}
\norm{h^*-h_\eta^*}_1\leq&\norm{F_0\hat h^*-F_0\hat h_\eta^*}_1+\norm{F_0\hat h_\eta^*-F_0^{app}\hat h_\eta^*}_1\\
&+\norm{Q\hat h-Q\hat h_n}_1+\norm{Q\hat h_n-Q^{app}\hat h_n}_1.
\end{align*}
We bound these in section \ref{a:PB} giving
\begin{enumerate}
\item $\norm{F_0\hat h^*-F_0\hat h_\eta^*}_1\leq$\datax{I}
\item $\norm{F_0\hat h_\eta^*-F_0^{app}\hat h_\eta^*}_1\leq$\datax{II}
\item $\norm{Q\hat h-Q\hat h_n}_1\leq$\datax{III}
\item $\norm{Q\hat h_n-Q^{app}\hat h_n}_1\leq$\datax{IV}
\end{enumerate}
This gives us $\norm{h^*-h^*_\eta}_1\leq$\datax{hetastar_pb_error}.
\subsubsection{Bounding Items (1) and (2)}\label{a:PB}
It is given in theorem \ref{thm:main} that 
\begin{equation*}\norm{F_0\hat h^*-F_0\hat h^*_\eta}_1\leq 2\norm{\hat h^*-\hat h^*_\eta}_1,\end{equation*} 
which we have from subsection \ref{ss:IPMResponse} is bounded by $2\cdot\text{\datax{hetastar_error}}=\text{\datax{I}}$.\marginnotet{some changes here}

We also have from theorem \ref{thm:main} that $\norm{F_0\hat h_\eta^*-F_0^{app}\hat h_\eta^*}_1$ is bounded by 
\begin{equation*}
\norm{\hat h^*_\eta}_{C^0}\frac{1}{2^{1-\gamma}(1-\gamma)}\sum_{\omega>N^*}\norm{g'_\omega}_{B}.\end{equation*}

We can compute explicitly $\norm{\hat h^*_\eta}_{C^0}$ and it is bounded by \datax{hetastar_max}, $\frac{1}{2^{1-\gamma}(1-\gamma)}\leq$\datax{gamma_const} and for $N^*=\datax{Nstar}$, $\sum_{\omega>N^*}\norm{g'_\omega}_B\leq$\datax{gprime_tail} as is shown in section \ref{a:gomegaprime}. 
These give us the bound $\norm{F_0\hat h_\eta^*-F_0^{app}\hat h_\eta^*}_1\leq$ \datax{II}.

\subsubsection{Bounding Items (3) and (4)}

We can bound $\norm{Q\hat h-Q\hat h_n}_1$ by
\begin{equation*}
\frac{1}{2^{1-\gamma}(1-\gamma)}\sum_{\omega}\norm{b_\omega}_{B}\norm{\hat h-\hat h_n}_{C^1}
\end{equation*}
for which we will need a bound on $\sum_{\omega}\norm{b_\omega}_{B}$.
In section \ref{a:bomega} we show is less than \datax{bomega_sum}; the bounds from earlier $\frac{1}{2^{1-\gamma}(1-\gamma)}\leq$\datax{gamma_const} and $\norm{\hat h-\hat h_n}_{C^1}\leq$\datax{heta_error} allow us to prove that 
\[
\norm{Q\hat h-Q\hat h_n}_1\leq\text{\datax{III}}.
\]

We now bound
\begin{equation*}
\norm{Q\hat h_n-Q^{app}\hat h_n}_1\leq \frac{1}{2^{1-\gamma}(1-\gamma)}\left[\sup_\omega|a_\omega|\cdot\norm{\hat h'_n}_{C^0}\cdot\sum_{|\omega|>N^*}\norm{g_\omega '}_B+\norm{\hat h_n}_{C^0}\sum_{\omega>N^*}\norm{b_\omega}_{B}\right].
\end{equation*}

We need a bound on $\sum_{\omega}|a_\omega|$ which we show is bounded by $\datax{aomega_max}$ in section \ref{a:aomega}, and the bounds from computer approximations of $\hat h_n$ which gives us $\norm{\hat h_n}_{C^0}\leq$\datax{heta_max} and $\norm{\hat h'_n}_{C^0}\leq$\datax{hetaprime_max}. 

We can use the same method from section \ref{a:bomega} to calculate $\sum_{\omega>N^*}\norm{b_\omega}_{B}\leq$ \datax{bomega_tail}. All this together gives us 
\[
\norm{Q\hat h_n-Q^{app}\hat h_n}_1\leq\text{\datax{IV}}.
\]

\begin{figure}
   \includegraphics[width=55mm]{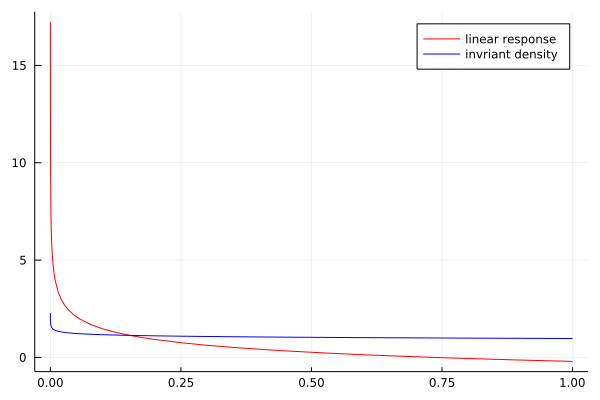}
   \caption{The linear response and invariant density of the LSV map for $\alpha=$\datax{alph}, calculated according to section \ref{ss:T} with $L^1$ error for the invariant density of \datax{hpb} and $L^1$ error for the linear response of \datax{hetastar_pb_error}.}
\end{figure}

\subsection{Normalizing the density and the linear response}
In this subsection we follow the estimates in Section \ref{sec:normalizing}.

First of all, we compute
\[
\norm{\frac{h_n}{\int h_n dm}-\frac{h}{\int h dm}}_1 \leq \textrm{\datax{normalised_pb}}.        
\]

Following through the calculations we bound \eqref{eq:finalbound}, the $L^1$ error on the normalized
linear response by \datax{normalisedstar_pb}.

\begin{figure}
   \includegraphics[width=55mm]{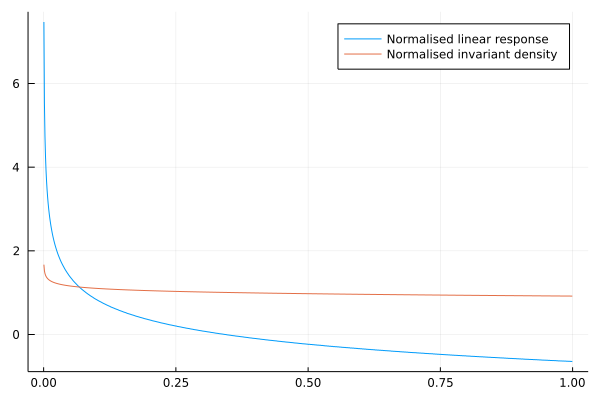}
   \caption{This shows the normalised invariant density of the LSV map at $\alpha=$\datax{alph} with the linear response of the normalised density.}
   \label{fig:Normalised}
\end{figure}

\section{Appendix: Effective bounds for \cite{K,BS}}\label{sec:appendix}
In this section we will use often the following notation following \cite{K}. Let $T_0$ be the left branch of the map $T$, let $z\in [0,1]$
and 
\[
z_r := T_0^{-r}(z).
\] 
By $(.)'$ we denote the derivative with respect to $z$.
To simplify the lookup of constants, they are presented in table \ref{tab:constants}.

\begin{table}[h!]
   \begin{center}
     \caption{Table of constants.}
     \label{tab:constants}
     \rotatebox{90}{
     \begin{tabular}{l|c|r} 
       \textbf{Label} & \textbf{Description} & \textbf{Value}\\
       \hline
       $\alpha$ & Parameter for the LSV map. & \datax{alph}\\
       $\gamma$ & Parameter for $\norm{\cdot}_B$. & \datax{gamma}\\
    $\delta_k$ & The size of the interval different from the true induced map. &\datax{delta_k}\\
       $\eta$ & Partition size for Ulam discretization. & \datax{eta}\\
    $N^*$ & The number of branches used to approximate operators $F$ and $Q$.& \datax{Nstar}\\
    $l^*$ & The number of itterations use to calculate $\hat h^*_\eta$.& \datax{lstar}\\
    $\nu$ & The number of Lasota-Yorke innequalities used to bound error in the Chebyshev projection.& \datax{Chebyshev_ν}\\
    $n$ & The number of Chebyshev polynomials used for the Chebyshev discretization.&\datax{Chebyshev_n}\\
    $C_1$ &$\frac{1}{1+\alpha2^\alpha}$&\datax{C_1}\\
    $C_2$ &$\frac{1}{\alpha(1-\alpha)2^{\alpha-1}}$&\datax{C_2}\\
    $C_3$ & $\frac{1}{\alpha}+\log{(C_1^{-1/\alpha})}$&\datax{C_3}\\
    $C_4$ &$\frac{1}{\alpha}-\frac{\log{(C_1^{1/\alpha}})}{\log{(2)}}$&\datax{C_4}\\
    $C_5$ &$2^\alpha$&\datax{C_5}\\
    $C_6$ &$(\alpha+1)2^\alpha$&\datax{C_6}\\
    $C_7$ &$\alpha(\alpha+1)2^\alpha$&\datax{C_7}\\
    $C_8$ & A computed value from section \ref{a:gomegaprime}. &\datax{C8}\\
    $C_{10}$ &$C_2\cdot C_5\cdot C_8$&\datax{C_10}\\
    $C_{11}$ &$C_2\cdot C_6\cdot C_8\frac{-\log {(1/C_2)+1}}{\alpha}$&\datax{C_11}\\
    $C_{12}$ &$2^\alpha C_2^2\cdot C_4\cdot C_7\cdot C_8\cdot $&\datax{C_12}\\
    $C_{sum}$ & $\frac{(\frac{1+\alpha-\gamma}{\gamma})^{-\gamma/\alpha+1/\alpha+1}(\alpha2^\alpha)^{-\gamma/\alpha}}{(\frac{1+\alpha-\gamma}{\gamma})^{1/\alpha+1}}$ & \datax{SummableConst}\\
    $D_0$ &A bound on the distortion of the branches of the induced map. &\datax{D0}\\
    $D$ &A bound on the distortion of the inverse of the branches of the induced map.&\datax{D}\\
     \end{tabular}
     }
   \end{center}
 \end{table}

\subsection{Estimating the tail $\sum_{\omega>N^*}\norm{g'_\omega}_{B}$}\label{a:gomegaprime}
For this we look at \cite{BS} lemma 5.2  which gives 
\begin{equation*}
\norm{g'_\omega}_B\leq C_8 \sup_{z\in (0,0.5]}z^\gamma (1+nz^\alpha \alpha 2^\alpha)^{-1/\alpha-1}.
\end{equation*}
using calculations from section \ref{a:bomega}
\begin{align*}
&\sup_{z\in (0,0.5]}z^\gamma (1+nz^\alpha \alpha 2^\alpha)^{-1/\alpha-1}\\
\leq&\sup_{z\in (0,0.5]}\frac{z^{\gamma-1-\alpha}(\alpha 2^\alpha)^{-1/\alpha-1}}{(z^{-\alpha}\alpha^{-1} 2^{-\alpha}+n)^{1/\alpha+1}}\\
\leq&C_{sum} n^{-\gamma/\alpha}\\
\end{align*}
so $\sum_{n>N^*}\norm{g'_\omega}_B\leq C_8 \cdot C_{sum}[\zeta(\gamma/\alpha)-\sum_{j=1}^{N^*}j^{-\gamma/\alpha}]$. 

The constant $C_8$ comes from \cite{K} where it is shown to be finite, but, when we calculate $C_8$ according to their proof we get 
$C_8=\exp{(1+(\alpha+1)^2 2^{2\alpha}C_2^2\frac{\pi^2}{6})}$, which is of order $10^{269}$.

Therefore we need a sharper bound for $C_8$. We start similarly
\begin{align*}
z_n'=&\Pi_{j=1}^n \frac{1}{1+(\alpha+1)2^\alpha z_j^\alpha}=\exp{(\sum_{j=1}^n -\log{(1+(\alpha+1)2^\alpha z_j^\alpha)})}\\
=&\exp{(\sum_{j=1}^n -(\alpha+1)2^\alpha z_j^\alpha+\sum_{j=1}^n[ -\log{(1+(\alpha+1)2^\alpha z_j^\alpha)+(\alpha+1)2^\alpha z_j^\alpha]})}\\
\leq&e\cdot (1+nz_0^\alpha \alpha 2^\alpha)^{-(\alpha+1)/\alpha}\cdot \exp{\sum_{j=1}^n[ -\log{(1+(\alpha+1)2^\alpha z_j^\alpha)+(\alpha+1)2^\alpha z_j^\alpha]})}
\end{align*}
where in the last line we use the calculation in \cite{K} following equation (5.7) which gives,
\begin{equation*}
-(\alpha+1)2^\alpha\sum_{j=1}^nz_j^\alpha\leq-\frac{\alpha+1}{\alpha}(\log{(1+nz_0^\alpha \alpha 2^\alpha)}+C
\end{equation*}
where $C$ comes from
\begin{equation*}
\sum_{j=1}^r\frac{1}{z_0^{-\alpha}+j\alpha2^\alpha}\geq \int_{1}^r \frac{z_0^\alpha}{1+tz_0^\alpha\alpha2^\alpha}dt-C.
\end{equation*}
Since the function in the integral is monotonically decreasing and $\frac{z_0^\alpha}{1+tz_0^\alpha\alpha2^\alpha}\leq1$ we can bound $C$ by $1$, which gives us the factor of $e$.
 
In the next paragraph we will use the Taylor expansion of $-\log(1+x)$, however this is only convergent for $x\in(-1,1)$, so first we choose a $j^*$ large enough that $\frac{-(\alpha+1)2^\alpha C_2}{j^*}\in(-1,1)$.
\begin{align*}
z_n'\leq &e\cdot (1+nz_0^\alpha \alpha 2^\alpha)^{-(\alpha+1)/\alpha}\cdot \exp{\sum_{j=1}^n[ -\log{(1+(\alpha+1)2^\alpha z_j^\alpha)+(\alpha+1)2^\alpha z_j^\alpha]})}\\
=&e\cdot (1+nz_0^\alpha \alpha 2^\alpha)^{-(\alpha+1)/\alpha}\cdot \exp{\sum_{j=1}^{j^*-1}[ -\log{(1+(\alpha+1)2^\alpha z_j^\alpha)+(\alpha+1)2^\alpha z_j^\alpha]})}\\
&\cdot \exp{\sum_{j=j^*}^n[ -\log{(1+(\alpha+1)2^\alpha z_j^\alpha)+(\alpha+1)2^\alpha z_j^\alpha]})}\\
\end{align*}

For $\exp{(\sum_{j=j^*}^n[ -\log{(1+(\alpha+1)2^\alpha z_j^\alpha)+(\alpha+1)2^\alpha z_j^\alpha]})}$ we use
\begin{align*}
&\exp{(\sum_{j=j^*}^n[ -\log{(1+(\alpha+1)2^\alpha z_j^\alpha)+(\alpha+1)2^\alpha z_j^\alpha]})}\\
\leq&\exp{(\sum_{j=j^*}^n[ -\log{(1+(\alpha+1)2^\alpha C_2 j^{-1})+(\alpha+1)2^\alpha  C_2 j^{-1}]})}.
\end{align*}
Substituting in the Taylor expansion of $\log(1+x)$ where $x=(\alpha+1)2^\alpha C_2j^{-1}$ gives
\begin{align*}
&\exp{(\sum_{j=j^*}^n\sum_{m=2}^\infty\frac{\left(-(\alpha+1)2^\alpha C_2\right)^m}{m}j^{-m})}\\
=&\exp{(\sum_{m=2}^\infty \sum_{j=j^*}^n\frac{\left(-(\alpha+1)2^\alpha C_2\right)^m}{m}j^{-m})}\\
=&\exp{(\sum_{m=2}^\infty\frac{\left(-(\alpha+1)2^\alpha C_2\right)^m}{m} \sum_{j=j^*}^n j^{-m})}\\
\leq&\exp{(\sum_{m=2}^\infty\frac{\left(-(\alpha+1)2^\alpha C_2\right)^m}{m}[\zeta(m)-\sum_{j=1}^{j^*}\frac{1}{j^{m}}])}\\
\leq&\exp{(\sum_{m=2}^\infty\frac{\left(-(\alpha+1)2^\alpha C_2\right)^m}{m}[\zeta(2)-\sum_{j=1}^{j^*}\frac{1}{j^{2}}])}\\
\leq&\exp{(-[\zeta(2)-\sum_{j=1}^{j^*}\frac{1}{j^2}]\cdot[\log{(1+(\alpha+1)2^\alpha C_2)}-(\alpha+1)2^\alpha C_2])}\\
=&\exp{((\alpha+1)2^\alpha C_2[\zeta(2)-\sum_{j=1}^{j^*}\frac{1}{j^2}])}\cdot(1+(\alpha+1)2^\alpha C_2)^{-[\zeta(2)-\sum_{j=1}^{j^*}\frac{1}{j^2}]}
\end{align*}

To get our final estimate we need to bound
\begin{equation*}
   \exp{\sum_{j=1}^{j^*-1}[ -\log{(1+(\alpha+1)2^\alpha z_j^\alpha)+(\alpha+1)2^\alpha z_j^\alpha]})},
\end{equation*}
Since there are a finite number of terms we can bound it from above through the use of rigorous numerical methods.

Choosing $j^*=$\datax{jstar} gives us that $\frac{-(\alpha+1)2^\alpha C_2}{j^*}=$\datax{C8jstarx}$\in (-1,1)$ and $C_8\geq$\datax{C8} is an upper bound.

This gives us that for $N^*=$\datax{Nstar}, $\sum_{\omega>N^*}\norm{g'_\omega}_B\leq$\datax{gprime_tail}.

\subsection{Bounds for lemma \ref{lem:tripleC2}}\label{a:DD0numbers}

We want a bound on $D_0=\norm{\frac{g_\omega''}{g_\omega'}}_\infty$. In \cite{K} they have bounds for $z_r$ where $g_{r+1}'=0.5z_r'$ and $g_{r+1}''=0.5z_r''$ so we may use the bound from \cite[Lemma 5.4]{K}  which gives
\begin{equation*}
\frac{z_r''}{z_r'}=\frac{\alpha(\alpha+1)2^\alpha z^{{\alpha-1}}_{r+1}z_{r+1}'}{1+(\alpha+1)2^{\alpha}z_{r+1}^\alpha}+\frac{z_{r+1}''}{z_{r+1}'}
\end{equation*}
which we can use to get a bound on $\frac{z_r''}{z_r'}-\frac{z_{r+1}''}{z_{r+1}'}$. We use lemma 5.2 and 5.3 from \cite{K} to get
\begin{equation*}
z_r'\leq C_8(\alpha 2^\alpha)^{-(\alpha+1)/\alpha}r^{-(\alpha+1)/\alpha}z_0^{-\alpha-1}
\end{equation*}
where we calculate $C_8$ in section \ref{a:gomegaprime}, and
\begin{equation*}
z_r^{\alpha-1}=(z_r^\alpha)^{(\alpha-1)/\alpha}\leq \left(\frac{2^{1-\alpha}}{\alpha(1-\alpha)}\right)^{(\alpha-1)/\alpha}r^{-(\alpha-1)/\alpha}
\end{equation*}
giving us
\begin{equation*}
z_r^{\alpha-1}z_r'\leq C_8 \frac{\alpha^{-2}2^{1-1/\alpha-2\alpha}}{(1-\alpha)^{(\alpha-1)/\alpha}}r^{-2}z_0^{-\alpha-1}.
\end{equation*}

Then
\begin{align*}
\frac{z_r''}{z_r'}-\frac{z_{r+1}''}{z_{r+1}'}\leq &\alpha(\alpha+1)2^\alpha C_8 \frac{\alpha^{-2}2^{1-1/\alpha-2\alpha}}{(1-\alpha)^{(\alpha-1)/\alpha}}r^{-2}z_0^{-\alpha-1}\\
\leq &(\alpha+1)C_8 \frac{\alpha^{-1}2^{1-1/\alpha-\alpha}}{(1-\alpha)^{(\alpha-1)/\alpha}}r^{-2}0.5^{-\alpha-1}
\end{align*}
from which follows that
\begin{equation*}
sup_r\norm{\frac{z_r''}{z_r'}}_\infty\leq \pi^2\frac{0.5^{-\alpha-1}}{6}(\alpha+1)C_8 \frac{\alpha^{-1}2^{1-1/\alpha-\alpha}}{(1-\alpha)^{(\alpha-1)/\alpha}}
\end{equation*}
which for $\alpha=\datax{alph}$ gives $D_0=\datax{D0}$.

Since
\begin{align*}
\log(\frac{g_\omega'(x)}{g_\omega'(y)})&=[\log(g_\omega'(\zeta))]'(x-y)\\
=&\frac{g_\omega''(\zeta)}{g_\omega'(\zeta)}(x-y)\leq\frac{g_\omega''(\zeta)}{g_\omega'(\zeta)},
\end{align*}
we have that $D\leq\exp(D_0)\leq$\datax{D}.

\subsubsection{Lasota-Yorke inequalitys for $\hat L_{\delta_k}$ and $\hat L$.}\label{a:Cstar}
We use some estimates from \cite{BGNN};
\begin{enumerate}
\item From proposition 7.2 we have $var(\hat Lf)\leq \lambda var(f)+B\norm{f}_{1}$ with $B:=\norm{\hat T''/(\hat T')^2}_\infty$ and $\lambda:=1/\inf_{x}|D_x\hat T|$;
\item From proposition 7.4 $\norm{\hat L^nf}_{C^1}\leq M\lambda^n\norm{f}_{C^1}+M^2\norm{f}_\infty$ with $M:=1+\frac{B}{1-\lambda}$;
\item From proposition7.6 $\norm{\hat L^nf}_{C^2}\leq M(\lambda^2)^n\norm{f}_{C^2}+D\norm{f}_{C^1}$ where $D:=\max\{3\frac{\lambda B M}{1-\lambda},3M\left(\frac{B}{1-\lambda}\right)^2+MZ\}+M\lambda+M^2$. $Z$ being 
\begin{equation*}
\frac{1}{1-\lambda^2}\left(\norm{\hat T'''/(\hat T')^3}_\infty+\frac{3\lambda}{1-\lambda}\norm{\hat T''/(\hat T')^2}_\infty\right).
\end{equation*}
\end{enumerate}
By the construction of $\hat T_{\delta_k}$ we know $1/\inf_{x}|D_x\hat T_{\delta_k}|=1/\inf_{x}|D_x\hat T|$, $\norm{\hat T_{\delta_k}''/(\hat T_{\delta_k}')^2}_\infty\leq \norm{\hat T''/(\hat T')^2}_\infty$ and $\norm{\hat T_{\delta_k}'''/(\hat T_{\delta_k}')^3}_\infty\leq \norm{\hat T'''/(\hat T')^3}_\infty$ so bounding these values for $\hat T$ gives us inequalities that are true for both.

We note that $\norm{\hat T''/(\hat T')^2}_\infty=\sup_{\omega}\norm{g_\omega''/(g_\omega')}_\infty=D_0$ which is calculated in section \ref{a:DD0numbers}. We can calculate $\norm{\hat T'''/(\hat T')^3}_\infty$ a similar way as follows,
\begin{align*}
\norm{\frac{T'''}{(T')^3}}_\infty=\norm{(\frac{g_\omega'''}{(g_\omega')^4}+3\frac{(g_\omega'')^2}{(g_\omega')^5})(g_\omega')^3}_\infty\leq \norm{\frac{g_\omega'''}{g_\omega'}}_\infty+3\norm{\left(\frac{g_\omega''}{g_\omega'}\right)^2}_\infty
\end{align*}
so we need to bound $\norm{\frac{g_\omega'''}{g_\omega'}}_\infty$ and $\norm{\left(\frac{g_\omega''}{g_\omega'}\right)^2}_\infty$, the second of which is $D_0^2$ from section \ref{a:DD0numbers}. In \cite{K} it is proven that $\norm{\frac{g_\omega'''}{g_\omega'}}_\infty$ is bounded and their method gives that it is less than
\begin{equation*}
\sum_{r=0}^\infty(\alpha-1)\alpha(\alpha+1)2^\alpha z_{r+1}^{\alpha-2}(z_{r+1}')^2+3\alpha(\alpha+1)2^\alpha z_{r+1}^{\alpha-1}z_{r+1}'\frac{z_{r+1}''}{z_{r+1}'}
\end{equation*}
where $g_\omega=z_r\circ g_1$, so $z_{r}'=\frac{g_\omega'}{2}$ and $z_r'''=\frac{g_\omega'''}{2}$. We have bounds on $\frac{z_{r+1}''}{z_{r+1}'}$ and $z_{r+1}^{\alpha-1}z_{r+1}'$ from section \ref{a:DD0numbers}, and we bound $z_{r+1}^{\alpha-2}$ using lemma 5.2 of \cite{K} to get $z_{r+1}^{\alpha-2}\leq (\frac{C_2}{r+1})^{(\alpha-2)/\alpha}$. We bound $(z_{r+1}')^2$ using lemma 5.3 of \cite{K} to get $(z_{r+1}')^2\leq C_8^2(2^\alpha \alpha)^{-2(\alpha+1)/\alpha}z_0^{-2(\alpha+1)}(r+1)^{-2(\alpha+1)/\alpha}$ so we can bound $\norm{\frac{g_\omega'''}{g_\omega'}}_\infty$ by
\begin{equation*}
\left(z_0^{-\alpha-4}\sum_{r=0}^\infty (r+1)^{-3}\right)[(\alpha-1)\alpha(\alpha+1)2^\alpha C_2^{(\alpha-2)/\alpha}C_8^2(2^\alpha \alpha)^{-2(\alpha+1)/\alpha}+3\alpha(\alpha+1)2^\alpha C]
\end{equation*}
where $C$ is the product of the values from section \ref{a:DD0numbers}. Substituting in the maximizing value of $z_0=0.5$ and note $\sum_{r=0}^\infty (r+1)^{-3}=\zeta(3)$ to get a bound of \datax{Cp01}.

This gives us a bound of $\norm{\frac{T'''}{(T')^3}}_\infty\leq$\datax{Cp} and we have
\begin{itemize}
\item $\lambda=0.5$,
\item $B\leq $\datax{TB},
\item $M\leq $\datax{TM},
\item $Z\leq $\datax{TZ},
\item $D\leq $\datax{TD}.
\end{itemize}
These values give us the explicit bounds
\begin{enumerate}
\item$\textrm{Var}(\hat L_{\delta_k}f)\leq 0.5 \textrm{Var}(f)+\datax{TB}\norm{f}_{1}$;\label{eq:LYL1}
\item $\norm{\hat L_{\delta_k}^nf}_{C^1}\leq \datax{TM}\cdot0.5^n\norm{f}_{C^1}+\datax{TM2}\norm{f}_\infty$;\label{eq:LYC1}
\item $\norm{\hat L_{\delta_k}^nf}_{C^2}\leq \datax{TM}\cdot0.25^n\norm{f}_{C^2}+\datax{TD}\norm{f}_{C^1}$.\label{eq:LYC2}
\end{enumerate}
These Lasota-Yorke inequalities give us the bounds $C^*=M\lambda+M^2=$\datax{Cstar} for lemma \ref{rem:approx1}.

For a bound on $\norm{\hat h_{\delta_k}}_{C^1}$ and $\norm{\hat h_{\delta_k}}_{C^1}$ we observe that $\norm{\hat h_{\delta_k}}_1=1$, and $\hat L_{\delta_k}\hat h_{\delta_k}=\hat h_{\delta_k}$, the inequalities above give us
\begin{enumerate}
\item $var(\hat h_{\delta_k})\leq \datax{TB}\implies \norm{\hat h_{\delta_k}}_{BV}\leq 1+\datax{TB}=$\datax{hDeltaBV}
\item $\norm{\hat h_{\delta_k}}_{C^1}\leq \datax{TM2}\norm{\hat h_{\delta_k}}_{\infty}\leq \datax{TM2}\norm{\hat h_{\delta_k}}_{BV}=$\datax{hDeltaC1}
\item $\norm{\hat h_{\delta_k}}_{C^2}\leq \datax{TD}\norm{\hat h_{\delta_k}}_{C^1}\leq $\datax{hDeltaC2}
\end{enumerate}

\subsection{Bounding $\sum_\omega \norm{\partial_\epsilon g_\omega'}_B$}\label{a:bomega}
For this we use from Lemma 5.2. \cite{BS}
\begin{align*}
\sum_\omega \norm{\partial_\epsilon g_\omega'}_B\leq &C_5\sum_{n=1}^\infty \sup_{\alpha\in U}\sup_{z\in (0,0.5]}|z^\gamma \cdot z_n'|\sum_{j=1}^n z_j^{\alpha}\\
+&C_6\sum_{n=1}^\infty \sup_{\alpha\in U}\sup_{z\in (0,0.5]}|z^\gamma \cdot z_n'|\sum_{j=1}^n z_j^{\alpha}|\log{z_j}|\\
+&C_7\sum_{n=1}^\infty \sup_{\alpha\in U}\sup_{z\in (0,0.5]}|z^\gamma \cdot z_n'|\sum_{j=1}^n z_j^{\alpha-1}|\partial_\alpha z_j|
\end{align*}
where $C_5=2^\alpha$, $C_6=(\alpha+1)2^\alpha$ and $C_7=\alpha(\alpha+1)2^\alpha$. We now use our bound from \ref{a:gomegaprime} get $C_8$ for
$z_n'\leq C_8(1+rz_0^\alpha2^\alpha)^{-\frac{\alpha+1}{\alpha}}$. We use this to bound $|z^\gamma \cdot z_n'|$ by $C_8 z^\gamma(1+nz_0^\alpha2^\alpha)^{-\frac{\alpha+1}{\alpha}}$.

We then use $z_n^\alpha<\frac{C_2}{n}$ to get $z_j^\alpha\leq C_2 j^{-1}$, $\log(z_j)\leq \frac{-\log{(1/{C_2})}+1}{\alpha}\log{j}$. Using the fact that $z_n^{\alpha-1}|\partial_\alpha z_n|=z_n^\alpha \frac{|\partial_\alpha z_n|}{z_n}$ and inequality (5.9) from \cite{K}
\begin{equation*}
z_n^{\alpha-1}|\partial_\alpha z_n|\leq  z_n^\alpha\sum_{j=1}^n 2^\alpha z_j^\alpha (-\log{(2z_j)})\leq 2^\alpha\frac{C_2^2}{n}\sum_{j=1}^nj^{-1}(-\log(2z_j))
\end{equation*}
We now use the following
\begin{equation*}
-\log(2z_j)\leq C_4 \log(j)
\end{equation*}
To get a value on $C_4$ we do the following,
\begin{align*}
\frac{-\log(2z_j)}{\log(j)}&\leq \frac{-\log(2(\frac{C_1}{j})^{1/\alpha}z_0)}{\log(j)}\leq \frac{-\log(2C_1^{\frac{1}{\alpha}}z_0)+\frac{1}{\alpha}\log(j)}{\log(j)}\\
&\leq \frac{1}{\alpha}-\frac{\log(2C_1^{\frac{1}{\alpha}}z_0)}{\log(j)}\leq \frac{1}{\alpha}-\frac{\log(2C_1^{\frac{1}{\alpha}}z_0)}{\log(2)}\\
&=C_4 \leq \text{\datax{C_4}}
\end{align*}
where we note that $\partial_\alpha z_j=0$ for $|\omega|=1$ and so $j=1$ this is still a valid bound for $z_j^{\alpha-1}|\partial_\alpha z_j|$.

which gives us $\sum_\omega \norm{b_\omega}_B\leq (I)+(II)+(III)$ where
\begin{align*}
(I)\leq &C_{10}\sum_{n=1}^\infty \sup_{\alpha\in U}\sup_{z\in (0,0.5]}z^\gamma (1+nz^\alpha \alpha 2^\alpha)^{-1/\alpha-1}\sum_{j=1}^n (j^{-1})\\
(II)\leq &C_{11}\sum_{n=1}^\infty \sup_{\alpha\in U}\sup_{z\in (0,0.5]}z^\gamma (1+nz^\alpha \alpha 2^\alpha)^{-1/\alpha-1}\sum_{j=1}^n (j^{-1}\log{j})\\
(III)\leq &C_{12}\sum_{n=1}^\infty \sup_{\alpha\in U}\sup_{z\in (0,0.5]}z^\gamma (1+nz^\alpha \alpha 2^\alpha)^{-1/\alpha-1}\sum_{j=1}^n (j^{-1}\sum_{k=1}^j k^{-1}\log{k})
\end{align*}
where $C_{10}=C_5 \cdot C_8\cdot C_2$, $C_{11}=C_6\cdot C_8\cdot C_2 \cdot \frac{-\log{(1/{C_2})}+1}{\alpha}$ and $C_{12}=2^\alpha C_4\cdot C_7\cdot C_8 \cdot  C_2^2$.

To bound $(I)$ we use that $\sum_{j=1}^nj^{-1}\leq 1+\log(n)$ to get
\begin{align*}
z^\gamma& (1+nz^\alpha \alpha 2^\alpha)^{-1/\alpha-1}\sum_{j=1}^n (j^{-1})\leq\frac{z^\gamma (1+\log{(n)})}{(1+nz^\alpha \alpha 2^\alpha)^{1/\alpha+1}}\\
\leq&\frac{z^{\gamma-\alpha(1/\alpha +1)}(\alpha 2^\alpha)^{-1/\alpha-1}(1+\log{(n)})}{(z^{-\alpha}\alpha^{-1} 2^{-\alpha}+n)^{1/\alpha+1}}\leq\frac{z^{\gamma-1-\alpha}(\alpha 2^\alpha)^{-1/\alpha-1}(1+\log{(n)})}{(z^{-\alpha}\alpha^{-1} 2^{-\alpha}+n)^{1/\alpha+1}}.
\end{align*}
We can use this to bound 
\begin{align*}
&\sum_{n=1}^\infty \sup_{\alpha\in U}\sup_{z\in (0,0.5]}z^\gamma (1+nz^\alpha \alpha 2^\alpha)^{-1/\alpha-1}(1+\log (n))\\
\leq&\sum_{n=1}^\infty \sup_{\alpha\in U}\sup_{z\in (0,0.5]}\frac{z^{\gamma-1-\alpha}(\alpha 2^\alpha)^{-1/\alpha-1}(1+\log{(n)})}{(z^{-\alpha}\alpha^{-1} 2^{-\alpha}+n)^{1/\alpha+1}}\\
=&\sum_{n=1}^\infty \sup_{\alpha\in U}\sup_{z\in (0,0.5]}\frac{z^{\gamma-1-\alpha}(\alpha 2^\alpha)^{-1/\alpha-1}}{(z^{-\alpha}\alpha^{-1} 2^{-\alpha}+n)^{1/\alpha+1}}+\frac{z^{\gamma-1-\alpha}(\alpha 2^\alpha)^{-1/\alpha-1}\log{(n)}}{(z^{-\alpha}\alpha^{-1} 2^{-\alpha}+n)^{1/\alpha+1}}
\end{align*}
In order to calculate bounds we must find the $z\in (0,0.5]$ that gives us the maximum value, which we do by finding the zero of the derivative of the part that depends on $z$,
\begin{align*}
&\partial_z \frac{z^{\gamma-1-\alpha}}{(z^{-\alpha}\alpha^{-1}2^{-\alpha}+n)^{1/\alpha+1}}\\
=&\frac{\partial_zz^{\gamma-1-\alpha}}{(z^{-\alpha}\alpha^{-1}2^{-\alpha}+n)^{1/\alpha+1}}+z^{\gamma-1-\alpha}\partial_z\frac{1}{(z^{-\alpha}\alpha^{-1}2^{-\alpha}+n)^{1/\alpha+1}}\\
=&\frac{(\gamma-1-\alpha)z^{\gamma-2-\alpha}}{(z^{-\alpha}\alpha^{-1}2^{-\alpha}+n)^{1/\alpha+1}}
+z^{\gamma-1-\alpha}\frac{-(1/\alpha+1)\cdot-\alpha\cdot z^{-\alpha-1}\cdot \alpha^{-1}2^{-\alpha}}{(z^{-\alpha}\alpha^{-1}2^{-\alpha}+n)^{1/\alpha+2}}\\
=&\frac{(\gamma-1-\alpha)z^{\gamma-2-\alpha}}{(z^{-\alpha}\alpha^{-1}2^{-\alpha}+n)^{1/\alpha+1}}
+\frac{(1+\alpha)\cdot z^{\gamma-2\alpha-2}\cdot \alpha^{-1}2^{-\alpha}}{(z^{-\alpha}\alpha^{-1}2^{-\alpha}+n)^{1/\alpha+2}}\\
=&\frac{z^{\gamma-\alpha-2}}{(z^{-\alpha}\alpha^{-1}2^{-\alpha}+n)^{1/\alpha+1}}\left((\gamma-1-\alpha)+\frac{(1+\alpha)z^{-\alpha}\alpha^{-1}2^{-\alpha}}{(z^{-\alpha}\alpha^{-1}2^{-\alpha}+n)}\right)
\end{align*}
which is zero when $-(\gamma-1-\alpha)=\frac{(1+\alpha)z^{-\alpha}\alpha^{-1}2^{-\alpha}}{(z^{-\alpha}\alpha^{-1}2^{-\alpha}+n)}$, we let $y=z^{-\alpha}\alpha^{-1}2^{-\alpha}$ which gives
\begin{align*}
&-(\gamma-1-\alpha)=(1+\alpha)\frac{y}{y+n}\\
\implies&-(\gamma-1-\alpha)n+(1+\alpha)y-\gamma y=(1+\alpha)y\\
\implies&-(\gamma-1-\alpha)n-\gamma y=0\\
\end{align*}
Therefore $y=\frac{(1+\alpha-\gamma)}{\gamma}n$ and 
\[
z=(\alpha 2^\alpha\frac{(1+\alpha-\gamma)}{\gamma}n)^{-1/\alpha}.
\]
We substitute this into the first sum
\begin{align*}
&\sum_{n=1}^\infty\sup_{\alpha\in U}\sup_{z\in (0,0.5]}\frac{z^{\gamma-1-\alpha}(\alpha 2^\alpha)^{-1/\alpha-1}}{(z^{-\alpha}\alpha^{-1} 2^{-\alpha}+n)^{1/\alpha+1}}\\
=&\sum_{n=1}^\infty\sup_{\alpha\in U}\frac{(\alpha 2^\alpha\frac{(1+\alpha-\gamma)}{\gamma}n)^{-\gamma/\alpha+1/\alpha+1}(\alpha 2^\alpha)^{-1/\alpha-1}}{(\alpha 2^\alpha\frac{(1+\alpha-\gamma)}{\gamma}n\alpha^{-1} 2^{-\alpha}+n)^{1/\alpha+1}},\\
\end{align*}
therefore
\begin{align*}
&\sum_{n=1}^\infty\sup_{\alpha\in U}\frac{(\frac{(1+\alpha-\gamma)}{\gamma})^{-\gamma/\alpha+1/\alpha+1}n^{-\gamma/\alpha}(\alpha 2^\alpha)^{-\gamma/\alpha}}{(\frac{1+\alpha-\gamma}{\gamma}+1)^{1/\alpha+1}}\\
=&\sup_{\alpha\in U}\frac{(\frac{(1+\alpha-\gamma)}{\gamma})^{-\gamma/\alpha+1/\alpha+1}(\alpha 2^\alpha)^{-\gamma/\alpha}}{(\frac{1+\alpha-\gamma}{\gamma}+1)^{1/\alpha+1}}\sum_{n=1}^\infty n^{-\gamma/\alpha}\\
=&\sup_{\alpha\in U}\frac{(\frac{(1+\alpha-\gamma)}{\gamma})^{-\gamma/\alpha+1/\alpha+1}(\alpha 2^\alpha)^{-\gamma/\alpha}}{(\frac{1+\alpha-\gamma}{\gamma}+1)^{1/\alpha+1}}\zeta (\gamma/\alpha).
\end{align*}
By the same calculation we have the second sum is bounded by
\begin{align*}
&\sum_{n=1}^\infty \sup_{\alpha\in U}\sup_{z\in (0,0.5]}\frac{z^{\gamma-1-\alpha}(\alpha 2^\alpha)^{-1/\alpha-1}\log{(n)}}{(z^{-\alpha}\alpha^{-1} 2^{-\alpha}+n)^{1/\alpha+1}}\\
\leq&\sup_{\alpha\in U}\frac{(\frac{(1+\alpha-\gamma)}{\gamma})^{-\gamma/\alpha+1/\alpha+1}(\alpha 2^\alpha)^{-\gamma/\alpha}}{(\frac{1+\alpha-\gamma}{\gamma}+1)^{1/\alpha+1}}\sum_{n=1}^\infty n^{-\gamma/\alpha}\log(n)\\
=&\sup_{\alpha\in U}\frac{(\frac{(1+\alpha-\gamma)}{\gamma})^{-\gamma/\alpha+1/\alpha+1}(\alpha 2^\alpha)^{-\gamma/\alpha}}{(\frac{1+\alpha-\gamma}{\gamma}+1)^{1/\alpha+1}}|\zeta'(\gamma/\alpha)|
\end{align*}
 which gives us
\begin{equation*}
(I) \leq C_{10}\cdot C_{sum} (\zeta(\gamma/\alpha)+|\zeta'(\gamma/\alpha)|).
\end{equation*}
where $C_{sum}=\sup_{\alpha\in U}\frac{(\frac{(1+\alpha-\gamma)}{\gamma})^{-\gamma/\alpha+1/\alpha+1}(\alpha 2^\alpha)^{-\gamma/\alpha}}{(\frac{1+\alpha-\gamma}{\gamma}+1)^{1/\alpha+1}}\leq$\datax{SummableConst}.

To bound $(II)$ we do the same, but using $\sum_{j=1}^n\log{(j)}j^{-1}\leq \log^2{(n)}$
\begin{align*}
&\sum_{n=1}^\infty \sup_{\alpha\in U}\sup_{z\in (0,0.5]}z^\gamma (1+nz^\alpha \alpha 2^\alpha)^{-1/\alpha-1}\sum_{j=1}^n (j^{-1}\log{j})\\
\leq&C_{sum}\sum_{n=1}^\infty n^{-\gamma/\alpha}\log^2{(n)}\\
\leq&C_{sum}\zeta''(\gamma/\alpha)
\end{align*}
giving
\begin{equation*}
(II) \leq C_{11}\cdot C_{sum}\zeta''(\gamma/\alpha).
\end{equation*}

For $(III)$ we use $\sum_{j=1}^n (j^{-1}\sum_{k=1}^j k^{-1}\log{k})\leq \sum_{j=1}^n (j^{-1}\log^2(j))\leq \log^3(n)$ to get
\begin{align*}
&\sum_{n=1}^\infty \sup_{\alpha\in U}\sup_{z\in (0,0.5]}z^\gamma (1+nz^\alpha \alpha 2^\alpha)^{-1/\alpha-1}\sum_{j=1}^n (j^{-1}\sum_{k=1}^j k^{-1}\log{k})\\
\leq&C_{sum}\sum_{n=1}^\infty n^{-\gamma/\alpha}\log^3{(n)}\\
\leq&C_{sum}|\zeta'''(\gamma/\alpha)|\\
\end{align*}
This implies directly that
\begin{equation*}
(III) \leq C_{12}\cdot C_{sum}\zeta'''(\gamma/\alpha).
\end{equation*}

In order to get the bound closer, we can use the tecnique of calculating the first $N$ terms of $\sum_{n=1}^\infty \norm{b_\omega}_B$ using the computer calculations from \ref{sec:A0B0} and the range estimation method from \cite{Tucker}.
\begin{align*}
\sum_\omega \norm{b_\omega}_B&\leq   C_{10}\cdot C_{sum}(\zeta(\gamma/\alpha)+|\zeta'(\gamma/\alpha)|-\sum_{j=1}^N[j^{-\gamma/\alpha}+j^{-\gamma/\alpha}\log{(j)}])\\
+&C_{11}\cdot C_{sum}(\zeta''(\gamma/\alpha)-\sum_{j=1}^N[j^{-\gamma/\alpha}\log^2{(j)}])\\
+&C_{12}\cdot C_{sum}(|\zeta'''(\gamma/\alpha)|-\sum_{j=1}^N[j^{-\gamma/\alpha}\log^3(j)])\\
+&\sum_{1\leq|\omega|\leq N}\norm{b_\omega}_B.\\
\end{align*}
We calculate upper bounds on the derivatives of $\zeta(x)$ using methods from \cite{C}.

Choosing $N=$\datax{Nstar} and $j^*=$\datax{jstar} gives us $\sum_\omega \norm{b_\omega}_B\leq$ \datax{bomega_sum}. The tail of the sum starting at $n=$\datax{Nstar} gives $\sum_{|\omega|\geq n} \norm{b_\omega}_B\leq$ \datax{bomega_tail}.

\subsection{Bounding $\sup_\omega |a_\omega|$}\label{a:aomega}
For $\sup_\omega |a_\omega|$ we use lemma 5.2 from \cite{K}. The proof of this lemma gives us
\begin{equation*}
\frac{z_0^\alpha}{n}\cdot\frac{1}{1+\alpha 2^\alpha}\leq z_n^\alpha\leq \frac{1}{z_0^{-\alpha}+n\alpha(1-\alpha)2^{\alpha-1}}
\end{equation*}
from which we get $C_1=\frac{1}{1+\alpha 2^\alpha}$ and $C_2=\frac{1}{\alpha(1-\alpha)2^{\alpha-1}}$
\begin{equation}\label{eq:Korepanov5.2}
\frac{z_0^\alpha}{n}\cdot C_1\leq z_n^\alpha\leq \frac{1}{n}\cdot C_2.
\end{equation}
Then $z_0^\alpha \frac{C_1}{n}\leq z_n^\alpha$ gives us $-\log{z_0}\leq\frac{-1}{\alpha}\log{\frac{C_1}{n}}-\log{z_0}$. To get a $C_3$ such that $\frac{-1}{\alpha}\log{\frac{C_1}{n}}\leq C_3 \logg{(n)}$ we take $C_3=\log{(C_1^{-1/\alpha})}+\frac{1}{\alpha}$. Since $C_3>1$
\begin{equation*}
-\log{z_0}\leq C_3( \logg{(n)}-\log{z_0}).
\end{equation*}

Then from the proof of lemma 5.2 from \cite{BS} we have
\begin{equation}\label{eq:Wael5.2}
\partial_\alpha z_{n+1}\leq2^\alpha\sum^{n+1}_{j=1}z_j^{\alpha+1}(-\log{2z_j})
\end{equation}
where $\sup_\omega |a_\omega|\leq \sup_{z_0\in [0,0.5]}\partial_\alpha z_{n+1}$. We use the fact that $x^{\alpha+1}(-\log{2x})$ is monotonicly increasing below $x=0.5\exp{(\frac{-1}{\alpha+1})}$ to say that if $C_2^{1/\alpha}{j^*}^{-1/\alpha}\leq0.5\exp{(\frac{-1}{\alpha+1})}$ then
\begin{align*}
\partial_\alpha z_{n+1}\leq&2^\alpha\sum^{n+1}_{j=1}z_j^{\alpha+1}(-\log{2z_j})\\
=&2^\alpha\sum^{n+1}_{j=j^*}z_j^{\alpha+1}(-\log{2z_j})+2^\alpha\sum^{j^*-1}_{j=1}z_j^{\alpha+1}(-\log{2z_j}).
\end{align*}
We may use a computer to calculate the sum up to $j^*-1$ and we bound the rest as follows,
\begin{align*}
&2^\alpha\sum^{n+1}_{j=j^*}z_j^{\alpha+1}(-\log{2z_j})\\
\leq&2^\alpha\sum^{n+1}_{j=j^*}{[C_2^{1/\alpha}j^{-1/\alpha}]}^{\alpha+1}(-\log{(2C_2^{1/\alpha}j^{-1/\alpha})})\\
=&2^\alpha\sum^{n+1}_{j=j^*}{[C_2^{1/\alpha}j^{-1/\alpha}]}^{\alpha+1}(-\log{(j^{-1/\alpha})}-\log{(2C_2^{1/\alpha})})\\
\leq&2^\alpha\sum^{n+1}_{j=j^*}{[C_2^{1/\alpha}j^{-1/\alpha}]}^{\alpha+1}(-\log{(j^{-1/\alpha})})\\
\leq&\frac{2^{\alpha} C_2^{(\alpha+1)/\alpha}}{\alpha}\sum^{n+1}_{j=j^*}j^{-1-1/\alpha}\log{j}.
\end{align*}
Noticing that $\sum^{\infty}_{j=1}j^{-1-1/\alpha}\log{j}=-\zeta'(1+1/\alpha)$ which can be calculated by methods from \cite{C} gives us
\begin{equation}
\sup_\omega |a_\omega|\leq \frac{[-\zeta'(1+1/\alpha)-\sum_{j=1}^{j^*-1}j^{-1-1/\alpha}\log{j}]2^{\alpha} }{\alpha(\alpha(1-\alpha)2^{\alpha-1})^{(\alpha+1)/\alpha}}+2^\alpha\sum^{j^*-1}_{j=1}z_j^{\alpha+1}(-\log{2z_j})
\end{equation}
which for $\alpha=$\datax{alph} and taking $j^*=$\datax{jstar} gives $\sup_\omega |a_\omega|\leq$\datax{aomega_max}

\section{Appendix: Computing derivatives}\label{a:ComputationalMethods}\label{sec:A0B0}

In order to calculate $A_0$, $B_0$, $a_\omega$ and $b_\omega$ we use an iterative formula. We start with \marginnotet{some changes here}
\begin{equation*}
g_\omega\circ T_\omega(x)=x
\end{equation*}
from which we get
\begin{align*}
(g_\omega\circ T_\omega)'(x)=&g_\omega'\circ T_\omega(x)\cdot T_\omega'(x)=1\\
\implies g_\omega'(x)=&\frac{1}{T_\omega'\circ g_\omega(x)}
\end{align*}
and
\begin{align*}
\partial_\alpha(g_\omega\circ T_\omega)(x)=&\partial_\alpha g_\omega \circ T_\omega(x)+ g_\omega'\circ T_\omega(x)\cdot \partial_\alpha T_\omega(x)=0\\
\implies \partial_\alpha g_\omega(x)=&-\frac{\partial_\alpha T_\omega\circ g_\omega(x)}{T_\omega'\circ g_\omega(x)}.
\end{align*}
We then use these to get
\begin{align*}
(g_\omega'\circ T_\omega)'(x)=&g_\omega''\circ T_\omega(x)\cdot T_\omega'(x)=-\frac{T_\omega''(x)}{(T_\omega'(x))^2}\\
\implies g_\omega''(x)=&-\frac{T_\omega''\circ g_\omega(x)}{(T_\omega'\circ g_\omega(x))^3}
\end{align*}
and
\begin{align*}
\partial_\alpha(g_\omega'\circ T_\omega)(x)=&\partial_\alpha g_\omega' \circ T_\omega(x)+ g_\omega''\circ T_\omega(x)\cdot \partial_\alpha T_\omega(x)=\partial_\alpha \frac{1}{T_\omega'(x)}\\
\implies \partial_\alpha g_\omega'(x)=&\frac{T_\omega''\circ g_\omega(x)\cdot \partial_\alpha T_\omega\circ g_\omega(x)}{(T_\omega'\circ g_\omega)^3}-\frac{\partial_\alpha T_\omega'\circ g_\omega(x)}{(T_\omega'\circ g_\omega(x))^2}.
\end{align*}

We already can calculate $g_\omega$ so we need to calculate $\partial_\alpha T_\omega$, $T_\omega'$, $\partial_\alpha T_\omega'$ and $T_\omega''$. Note that $T_\omega = T_0^{n}\circ T_1$ where $|\omega|=n$, so
\begin{align*}(T_0^{n}\circ T_1)'=&(T_0^{n})'\circ T_1\cdot T_1'\\
\partial_\alpha(T_0^{n}\circ T_1)=&\partial_\alpha (T_0^{n})\circ T_1 + (T_0^{n})'\circ T_1\cdot \partial_\alpha T_1\\
(T_0^{n}\circ T_1)''=&(T_0^{n})''\circ T_1\cdot (T_1')^2+(T_0^{n})'\circ T_1\cdot T_1''\\
\partial_\alpha(T_0^{n}\circ T_1)'=&\partial_\alpha (T_0^{n})'\circ T_1\cdot T_1' + (T_0^{n})''\circ T_1\cdot T_1'\cdot \partial_\alpha T_1+(T_0^{n})'\circ T_1\cdot \partial_\alpha T_1'\\
(T_0^{n}\circ T_1)'''=&(T_0^n)'''\circ T_1\cdot (T_1')^3+3(T_0^n)''\circ T_1\cdot T_1'\cdot T_1''+(T_0^{n})'\circ T_1\cdot T_1'''\\
\partial_\alpha(T_0^{n}\circ T_1)''=&\partial_\alpha (T_0^{n})''\circ T_1\cdot (T_1')^2 +\partial_\alpha (T_0^{n})'\circ T_1\cdot T_1'' +(T_0^{n})'''\circ T_1\cdot (T_1')^2\cdot \partial_\alpha T_1\\
&+(T_0^{n})''\circ T_1\cdot (T_1''\cdot \partial_\alpha T_1+2\cdot T_1'\cdot \partial_\alpha T_1')+(T_0^{n})'\circ T_1\cdot \partial_\alpha T_1''
\end{align*}
which we may write as a matrix
\[
\begin{pmatrix}
T_\omega'\\
\partial_\alpha T_\omega\\
T_\omega''\\
\partial_\alpha T_\omega'\\
T_\omega'''\\
\partial_\alpha T_\omega''
\end{pmatrix}=
\begin{pmatrix}
T_1' & 0 & 0 & 0 & 0 & 0\\
\partial_\alpha T_1 & 1 & 0 & 0 & 0 & 0\\
T_1''& 0 & (T_1')^2 & 0 & 0 & 0\\
\partial_\alpha T_1' & 0 & T_1'\partial_\alpha T_1 & T_1' & 0 & 0\\
T_1''' & 0 & 3T_1''\cdot T_1' & 0 & T_1'^3 & 0\\
\partial_\alpha T_1'' & 0 & T_1''\cdot \partial_\alpha T_1+2\cdot T_1'\cdot \partial_\alpha T_1' & T_1'' & (T_1')^2\cdot \partial_\alpha T_1 & (T_1')^2
\end{pmatrix}\cdot
\begin{pmatrix}
(T_0^{n})'\circ T_1\\
\partial_\alpha(T_0^{n})\circ T_1\\
(T_0^{n})''\circ T_1\\
\partial_\alpha(T_0^{n})'\circ T_1\\
(T_0^{n})'''\circ T_1\\
\partial_\alpha (T_0^{n})''\circ T_1
\end{pmatrix}.
\]
By the same logic we may write
\[
\begin{pmatrix}
T_\omega'\\
\partial_\alpha T_\omega\\
T_\omega''\\
\partial_\alpha T_\omega'\\
T_\omega'''\\
\partial_\alpha T_\omega''
\end{pmatrix}=
\begin{pmatrix}
T_1' & 0 & 0 & 0 & 0 & 0\\
\partial_\alpha T_1 & 1 & 0 & 0 & 0 & 0\\
T_1''& 0 & (T_1')^2 & 0 & 0 & 0\\
\partial_\alpha T_1' & 0 & T_1'\partial_\alpha T_1 & T_1' & 0 & 0\\
T_1''' & 0 & 3T_1''\cdot T_1' & 0 & T_1'^3 & 0\\
\partial_\alpha T_1'' & 0 & T_1''\cdot \partial_\alpha T_1+2\cdot T_1'\cdot \partial_\alpha T_1' & T_1'' & (T_1')^2\cdot \partial_\alpha T_1 & (T_1')^2
\end{pmatrix}\]
\[
\cdot
\begin{pmatrix}
T_0'\circ T_1 & 0 & 0 & 0 & 0 & 0\\
\partial_\alpha T_0\circ T_1 & 1 & 0 & 0 & 0 & 0\\
T_0''\circ T_1& 0 & (T_0'\circ T_1)^2 & 0 & 0 & 0\\
\partial_\alpha T_0'\circ T_1 & 0 & T_0'\circ T_1\partial_\alpha T_0\circ T_1 & T_0'\circ T_1 & 0 & 0\\
T_0'''\circ T_1 & 0 & 3T_0''\circ T_1\cdot T_0'\circ T_0 & 0 & (T_0'\circ T_1)^3 & 0\\
\partial_\alpha T_0''\circ T_1 & 0 & T_0''\circ T_1\cdot \partial_\alpha T_0\circ T_1+2\cdot T_0'\circ T_1\cdot \partial_\alpha T_0'\circ T_1 & T_0''\circ T_1 & (T_0'\circ T_1)^2\cdot \partial_\alpha T_0\circ T_1 & (T_0'\circ T_1)^2
\end{pmatrix}\]
\[
\cdot
\begin{pmatrix}
(T_0^{n-1})'\circ T_0\circ T_1\\
\partial_\alpha(T_0^{n-1})\circ T_0\circ T_1\\
(T_0^{n-1})''\circ T_0\circ T_1\\
\partial_\alpha(T_0^{n-1})'\circ T_0\circ T_1\\
(T_0^{n-1})'''\circ T_0\circ T_1\\
\partial_\alpha(T_0^{n-1})''\circ T_0\circ T_1
\end{pmatrix}.
\]
and use induction to give a series of matrices such that
\[
\begin{pmatrix}
T_\omega'\\
\partial_\alpha T_\omega\\
T_\omega''\\
\partial_\alpha T_\omega'\\
T_\omega'''\\
\partial_\alpha T_\omega''
\end{pmatrix}=
\begin{pmatrix}
T_1' & 0 & 0 & 0 & 0 & 0\\
\partial_\alpha T_1 & 1 & 0 & 0 & 0 & 0\\
T_1''& 0 & (T_1')^2 & 0 & 0 & 0\\
\partial_\alpha T_1' & 0 & T_1'\partial_\alpha T_1 & T_1' & 0 & 0\\
T_1''' & 0 & 3T_1''\cdot T_1' & 0 & T_1'^3 & 0\\
\partial_\alpha T_1'' & 0 & T_1''\cdot \partial_\alpha T_1+2\cdot T_1'\cdot \partial_\alpha T_1' & T_1'' & (T_1')^2\cdot \partial_\alpha T_1 & (T_1')^2
\end{pmatrix}\]
\[
\cdot
\begin{pmatrix}
T_0'\circ T_1 & 0 & 0 & 0 & 0 & 0\\
\partial_\alpha T_0\circ T_1 & 1 & 0 & 0 & 0 & 0\\
T_0''\circ T_1& 0 & (T_0'\circ T_1)^2 & 0 & 0 & 0\\
\partial_\alpha T_0'\circ T_1 & 0 & T_0'\circ T_1\partial_\alpha T_0\circ T_1 & T_0'\circ T_1 & 0 & 0\\
T_0'''\circ T_1 & 0 & 3T_0''\circ T_1\cdot T_0'\circ T_0 & 0 & (T_0'\circ T_1)^3 & 0\\
\partial_\alpha T_0''\circ T_1 & 0 & T_0''\circ T_1\cdot \partial_\alpha T_0\circ T_1+2\cdot T_0'\circ T_1\cdot \partial_\alpha T_0'\circ T_1 & T_0''\circ T_1 & (T_0'\circ T_1)^2\cdot \partial_\alpha T_0\circ T_1 & (T_0'\circ T_1)^2
\end{pmatrix}\]
\[
\dots
\begin{pmatrix}
T_0'\circ T_0^{n-1}\circ T_1\\
\partial_\alpha(T_0)\circ T_0^{n-1}\circ T_1\\
(T_0)''\circ T_0^{n-1}\circ T_1\\
\partial_\alpha(T_0)'\circ T_0^{n-1}\circ T_1\\
(T_0)'''\circ T_0^{n-1}\circ T_1\\
\partial_\alpha(T_0)''\circ T_0^{n-1}\circ T_1
\end{pmatrix}.
\]
Using
\begin{align*}
&T_0 = x(1+(2x)^\alpha)\\
&T_1 = 2x-1\\
&T_0' = 1+(1+\alpha)(2x)^\alpha\\
&T_1' = 2\\
&\partial_\alpha T_0 = (\log(x)+\log(2))2^\alpha x^{\alpha+1}\\
&\partial_\alpha T_1 = 0\\
&T_0'' = \alpha (1+\alpha) 2^\alpha x^{\alpha-1}\\
&T_1'' = 0\\
&\partial_\alpha T_0' = (2x)^\alpha ((\alpha+1)(\log(x)+\log (2))+1)\\
&\partial_\alpha T_1' = 0\\
&T_0''' =(\alpha-1) \alpha (\alpha+1) 2^\alpha x^{\alpha-2}\\
&T_1''' = 0\\
&\partial_\alpha T_0'' = 2^\alpha x^{\alpha-1}(\alpha^2\log(2)+(\alpha+1)\alpha\log(x)+\alpha(2+\log(2))+1)\\
&\partial_\alpha T_1'' = 0
\end{align*}
we are able to calculate explicitly the values $A_0$, $B_0$, $a_\omega$ and $b_\omega$. To calculate $a_\omega$ and $b_\omega$ we use $g_\omega=g_1\circ g_0^{n-1}$ and we use $T_0^{m}\circ T_1\circ g_1\circ g_0^{n-1}= g_0^{n-1-m}$ to calculate (leaving the last two rows and collumns for brevity)
\[
\begin{pmatrix}
T_\omega'\circ g_\omega\\
\partial_\alpha T_\omega\circ g_\omega\\
T_\omega''\circ g_\omega\\
\partial_\alpha T_\omega'\circ g_\omega
\end{pmatrix}=
\begin{pmatrix}
T_1'\circ g_\omega & 0 & 0 & 0\\
\partial_\alpha T_1\circ g_\omega & 1 & 0 & 0\\
T_1''\circ g_\omega& 0 & (T_1')^2\circ g_\omega & 0\\
\partial_\alpha T_1'\circ g_\omega & 0 & T_1'\circ g_\omega\partial_\alpha T_1\circ g_\omega & T_1'\circ g_\omega
\end{pmatrix}\\
\]
\[\cdot
\begin{pmatrix}
T_0'\circ g_0^{n} & 0 & 0 & 0\\
\partial_\alpha T_0\circ g_0^{n} & 1 & 0 & 0\\
T_0''\circ g_0^{n}& 0 & (T_0'\circ g_0^{n})^2 & 0\\
\partial_\alpha T_0'\circ g_0^{n} & 0 & T_0'\circ g_0^{n}\partial_\alpha T_0\circ g_0^{n} & T_0'\circ g_0^{n}
\end{pmatrix}\dots
\begin{pmatrix}
T_0'\circ  g_0\\
\partial_\alpha T_0\circ g_0\\
T_0''\circ g_0\\
\partial_\alpha T_0'\circ g_0
\end{pmatrix}
\]
where we calculate $g_0^{m}$ using the shooting method from section \ref{ShootingMethod}.

\end{document}